\documentclass[12pt]{extarticle}
\usepackage{amsmath, amsthm, amssymb, hyperref, color}
\usepackage[shortlabels]{enumitem}
\usepackage{graphicx}
\usepackage{makecell}
\usepackage[final]{pdfpages}
\setboolean{@twoside}{false}
\usepackage{pdfpages}
\usepackage{tabularx}
\usepackage{caption}
\usepackage{subcaption}
\usepackage{scalefnt}
\usepackage{verbatim}
\oddsidemargin \evensidemargin
\newtheorem{theorem}{Theorem}
\numberwithin{theorem}{section}
\newtheorem{proposition}[theorem]{Proposition}
\newtheorem{lemma}[theorem]{Lemma}
\newtheorem{corollary}[theorem]{Corollary}

\newtheorem{remark}[theorem]{Remark}
\newtheorem{example}[theorem]{Example}
\newtheorem{conjecture}[theorem]{Conjecture}

\newcommand{\RR}{\mathbb{R}}
\newcommand{\QQ}{\mathbb{Q}}
\newcommand{\PP}{\mathbb{P}}
\newcommand{\CC}{\mathbb{C}}
\newcommand{\ZZ}{\mathbb{Z}}

  \date{}

\title{\textbf{Sixty-Four Curves of Degree Six}}

\author{Nidhi Kaihnsa, Mario Kummer, Daniel Plaumann, \\
         Mahsa Sayyary Namin, and Bernd Sturmfels}

\begin{document}
\maketitle

\begin{abstract} \noindent We present a computational study of smooth
  curves of degree six in the real projective plane. In the
  Rokhlin--Nikulin classification, there are $56$ topological types,
  refined into $64$ rigid isotopy classes.  We developed software that
  determines the topological type of a given sextic and used it to
  compute empirical probability distributions on the various types.
  We list $64$ explicit representatives with integer coefficients, and
  we perturb these to draw many samples from each class.  This allows
  us to explore how many of the bitangents, inflection points and
  tensor eigenvectors are real.  We also study the real tensor rank,
  the construction of quartic surfaces with prescribed topology, and
  the avoidance locus, which is the locus of all real lines that do
  not meet a given sextic. This is a union of up to $46$ convex
  regions, bounded by the dual curve.
\end{abstract}

\section{Topology} \label{1}

A classical theme in real algebraic geometry is the
topological classification of algebraic curves in the real projective plane $\PP^2_\RR$.
Hilbert's 16th problem asks for the {\em topological types} of smooth curves of 
a fixed degree $d$, where two curves $C$ and $C'$
have the same type if some homeomorphism of $\PP^2_\RR \rightarrow \PP^2_\RR$
restricts to a homeomorphism $C_\RR \rightarrow C'_\RR$. 
A finer notion of equivalence comes from the {\em discriminant}
$\Delta$.  This is an irreducible hypersurface of degree $3(d-1)^2$ in
the projective space $\PP^{d(d+3)/2}_\RR$ of curves of degree $d$.
Points on $\Delta$ are singular curves.  The {\em rigid
  isotopy classes} are the connected components of the complement
$\PP^{d(d+3)/2}_\RR \backslash \Delta$. If two curves $C$ and $C'$
are in the same rigid isotopy class, then they have the same
topological type, but the converse is not true in general, as shown by Kharlamov \cite{Kha}.

Hilbert's 16th problem has been solved up to $d  = 7$,
thanks to contributions by many mathematicians,
including Hilbert \cite{Hil},  Rohn \cite{Rohn}, Petrovsky \cite{Pet},
Rokhlin \cite{Rokh},
Gudkov \cite{Gud}, Nikulin \cite{Nik}, Kharlamov \cite{Kha}, and Viro \cite{Viro1, Viroadd, Viro2}.
In this article we focus on the case $d=6$.
 Our point of departure is the following well-known classification of sextics,
 found in    \cite[\S 7]{Viroadd}.

\begin{theorem}[Rokhlin--Nikulin] \label{thm:departure}
The discriminant of plane sextics is a hypersurface of degree $75$ in $\PP^{27}_\RR$
whose complement has $64$ connected components.
The $64$ rigid isotopy types are grouped into $56$ topological types,
with the number of ovals ranging from $0$ to $11$. The distribution is shown in 
Table \ref{tab:counts}.
The 56 types form the partially ordered set in 
Figure~\ref{fig:eins}. 
\end{theorem}

The $64$ types in Theorem \ref{thm:departure} were known to 
Rokhlin \cite{Rokh}. The classification was completed 
by Nikulin. It first appeared in his
paper \cite{Nik} on the arithmetic of real K3 surfaces.

\begin{table}[t]
 \begin{tabular}{|l|c|c|c|c|c|c|c|c|c|c|c|c||c|}
  \hline Number of ovals &0&1&2&3&4&5&6&7&8&9&10&11 & {\rm total} \\\hline
   Count of rigid isotopy types &1&1&2&4&4&7&6&10&8&12&6&3 & 64 \\\hline
  Count of topological types &1&1&2&4&4&5&6&7&8&9&6&3 & 56 \\\hline
 \end{tabular}
 \caption{\label{tab:counts} Rokhlin--Nikulin classification of smooth sextics in the real projective plane}
\end{table}

     Every connected component $C_0$ of a smooth curve
     $C$ in $\PP^2_\RR$ is homeomorphic to a circle. If the complement
     $\PP^2_\RR \backslash C_0$ is disconnected, then $C_0$ is called
     an {\em oval}, otherwise a {\em pseudoline}. If $C$ has even
     degree, then all connected components are ovals. A curve of odd
     degree has exactly one pseudoline. 
     The two connected components of the complement of an oval are called
     the {\em inside} and the {\em outside}. The former is homeomorphic
     to a disk, the latter to  a M\"obius strip.     An oval $C_0$
     {\em contains} another oval $C_1$ if $C_1$ lies in the inside
     of $C_0$. In that case, $C_0$ and $C_1$ are {\em nested ovals}.
     An oval is {\em empty} if it contains no other
     oval. The topological type of the curve $C$ is determined by the number of
     ovals together with the information of how these ovals contain
     each other.
We denote the type of a smooth plane sextic $C$ in~$\PP_\RR^2$~by
\begin{enumerate}
 \item[$\bullet$]  $\,\,k\,\,\,$ if $C$ consists of $k$ empty ovals;
 \vspace{-0.1in}
 \item[$\bullet$] $\,(k1)\,l\,\,$ if $C$ consists of an oval $C_0$
   containing $k$ empty ovals and of $l$ further empty ovals lying
   outside $C_0$;
 \vspace{-0.1in}
 \item[$\bullet$]   $\,({\rm hyp})\,\,$ if $C$ consists of three
   nested ovals (where 'hyp' stands for 'hyperbolic').
\end{enumerate}

\begin{figure} 
$ \!\!\!\!\!\!\!\!\!\! $ \includegraphics[height=20cm]{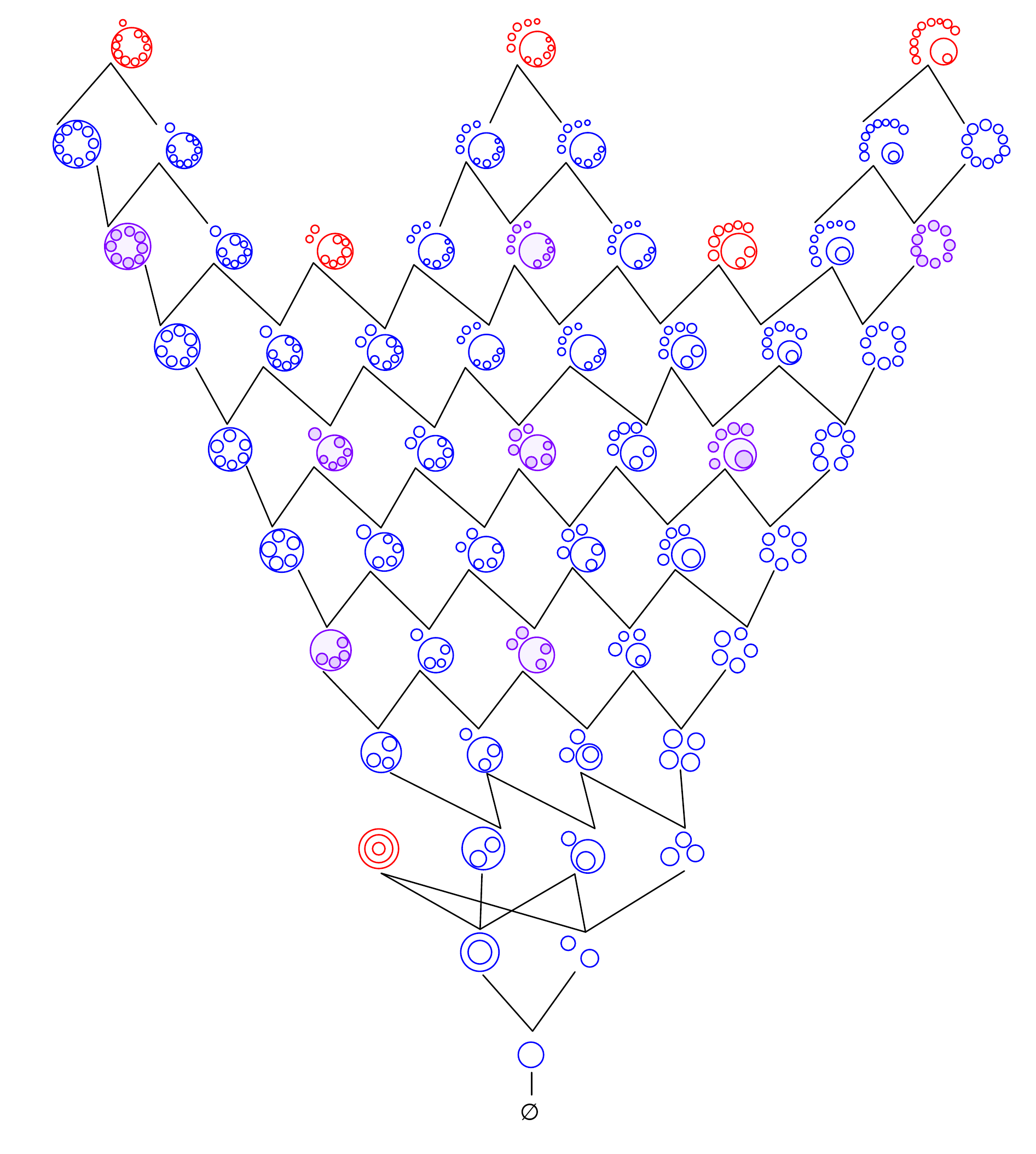}
\vspace{-0.3in}
    \caption{\label{fig:eins} The  56 types of smooth plane sextics form a partially ordered set.  
    The color code indicates whether the real curve divides its Riemann surface.
    The red curves are dividing,
    the blue curves are non-dividing, and the purple curves
    can be either dividing or non-dividing.}
     \end{figure}

Figure~\ref{fig:eins}  is a refinement of Viro's diagram in \cite[Figure 4]{Viro2}.
The top row contains the three types
with $11$ ovals. These are denoted $(91)1$, $(51)5$ and $(11)9$.
The cover relations correspond to either fusing two ovals or
shrinking an oval until it vanishes (cf.~Theorem \ref{thm:itenberg}).
We know from \cite[p.~107]{Nik} that all but eight of the $56$ topological types correspond to exactly one rigid isotopy class, 
while the following eight types consist of two rigid isotopy classes:
\begin{equation}
\label{eq:8areboth}
(41) \quad (21)2 \qquad 
(51)\, 1 \qquad (31)3 \qquad (11)5 \qquad 
(81) \qquad (41)4 \qquad 9 . \qquad
\end{equation}
For an irreducible curve $C$ in $\PP^2_\RR$, the set
$C_\mathbb{C}\backslash C_\mathbb{R}$ of non-real points in the
Riemann surface has either one or two connected components. In the
latter case we say that $C$ is \textit{dividing}.
 
The rigid isotopy type of a smooth plane sextic determines whether it
is dividing or not. More precisely,                                                                                                                                                                                                                                                                                                                                                                                                                                                                                                                                                                                        the subset of $\PP^{27}_\RR$ consisting of all dividing sextics is the
 closure of the union of $14$ rigid isotopy types. First, there
 are the eight topological types in  (\ref{eq:8areboth}). These can be dividing or not,
 by    \cite[Remark 3.10.10]{Nik}. This explains the difference between 
   the second row and third row of the table in Theorem \ref{thm:departure}.
   In addition, there are six topological types that are necessarily dividing.
   These are precisely the maximal elements in the poset of sextics:
   \begin{equation}
   \label{eq:6aredividing}
 (91)1 \qquad (51)5 \qquad (11)9 \qquad (61)2 \qquad (21)6 \qquad ({\rm hyp}). \qquad
\end{equation}

In summary, of the $56$ topological types of smooth plane sextics,
precisely $42$ types are non-dividing. The six types in (\ref{eq:6aredividing}) are dividing,
and the eight types in (\ref{eq:8areboth}) can be dividing or non-dividing.
Hence, there are $14$ rigid isotopy types that are dividing.
This accounts for all $\,64$ rigid isotopy types (connected components
of $\PP^{27}_\RR \backslash \Delta$) in the census of
Theorem~\ref{thm:departure}.

This paper presents an experimental study of the objects above,
conducted with a view towards Applied Algebraic Geometry.
Numerous emerging applications, notably in the analysis of 
data from the life sciences, now rely
on computational tools from topology and algebraic geometry. 
A long-term goal that motivated this project is the 
development of connections between such applications
and existing knowledge on  the topology of real algebraic varieties.

The ecosystem to be explored in this particular study is the $27$-dimensional space
of plane sextic curves.
Our focus lies on experimentation and exact computation with ternary sextics
over the integers. Thus, our model organisms are
homogeneous polynomials in $\ZZ[x,y,z]_6$.

We now summarize the contributions of this paper.
In Section 2 we present a list of $64$ polynomials in $\ZZ[x,y,z]_6$
that serve as representatives for the $64$ rigid isotopy types.
In Section 3 we describe methods for classifying
a  given ternary sextic according to Theorem~\ref{thm:departure}.
 To determine the topological type we wrote fast code in {\tt Mathematica}
 based on the built-in tool for 
 {\em cylindrical algebraic decomposition} (CAD, \cite{CWB}).
 This is used to sample sextics from natural probability distributions on
$\RR[x,y,z]_6 \simeq \RR^{28}$, so as to find the empirical
distributions on the $56$ topological types.
  Distinguishing between dividing and non-dividing types is harder.
Our primary tool for this is Proposition \ref{prop:turningout}.
For an alternative approach see \cite{kalla}.
  
  In Section 4 we present a method for computing
    the discriminant, and we
  discuss our derivation of the $64$ polynomial representatives.
    Section 5 concerns the subdivision of the dual
  projective plane $(\PP^2)^\vee_\RR$ by a curve 
  $C^\vee$ of degree $30$, namely   that dual
  to a given sextic $C$.  The nodes of $C^\vee$ are the $324$
  bitangents of $C$. We study how many of them are real for
  each of the $64$ types. These numbers do not depend on
  the topological or rigid isotopy type alone.
  They are reported in Table~\ref{tab:flexetc}.
    Real lines that miss
  $C_\RR$ form the avoidance locus $\mathcal{A}_C$. This is a union of up to $46$
  convex regions, bounded by the dual curve.  In
  Section 6 we explore inflection points, tensor
  eigenvectors, real tensor rank, and connections to K3 surfaces.

Many new results and questions can be derived
by the computational framework developed in this paper.
Here is one example. It concerns 
reducible sextic curves consisting of six distinct lines.
This $12$-dimensional family in $\PP^{27}_\RR$
is the {\em Chow variety}
of factorizable forms.

\begin{proposition} \label{prop:thirtyfive}
Configurations of six general lines  appear in the closure of precisely $35$ 
of the $64$ rigid isotopy classes.
These are the classes that meet the Chow variety in a generic point.
 These $35$ classes are marked with an asterisk in Table~\ref{tab:flexetc},
 in the column on eigenvectors.
\end{proposition}

\section{Representatives}

We shall derive the following result by furnishing explicit polynomial representatives.

\begin{proposition}
\label{prop:bigintegers}
Each of the $64$ rigid isotopy types can be realized by
a ternary sextic in $\ZZ[x,y,z]_6$ whose integer coefficients
have absolute value less than $1.5  \times 10^{38}$.
\end{proposition}

We list $64$ polynomials with integer coefficients
that represent the $64$ rigid isotopy types of smooth sextic curves 
in $\PP^2_\RR$. 
This list is available in a computer-algebra-friendly format at
\begin{equation}
\label{eq:url}
 \hbox{\url{http://personal-homepages.mis.mpg.de/kummer/sixtyfour.html}} 
 \end{equation}
 Each of our sextics is labeled by its topological type
and whether it is dividing (d) or non-dividing (nd)
in its Riemann surface. We begin
with the $35$ types that have at most $7$~ovals:

$$ \begin{small}
\begin{matrix}
    0 & {\rm nd} & x^6+y^6+z^6  \\
     1 & {\rm nd} &  x^6+y^6-z^6 \\      
     (11)  & {\rm nd} &  6(x^4+y^4-z^4)(x^2+y^2-2z^2)+x^5y \\
     2  & {\rm nd} &   (x^4 + y^4 - z^4) ((x + 4z)^2 + (y + 4z)^2 - z^2) + z^6 \\
     (21)  & {\rm nd} &  16 ((x {+} z)^2 + (y {+} z)^2 - z^2) (x^2 + y^2 - 
    7z^2) ((x {-} z)^2 + (y {-} z)^2 - z^2) + x^3y^3 \\
     (11)1 & {\rm nd} &  ((x + 2z)^2 + (y + 2z)^2 - z^2) (x^2 + y^2 - 3z^2) (x^2 + y^2 - 
    z^2) + x^5y  \\
     3  & {\rm nd} &  (x^2 + y^2 - z^2)(x^2 + y^2 - 2z^2)(x^2 + y^2 - 3z^2) + x^6  \\
     ({\rm hyp}) & {\rm d} & 6 (x^2 + y^2 - z^2)(x^2 + y^2 - 2z^2)(x^2 + y^2 - 3 z^2) + x^3 y^3 \\
          (31)  & {\rm nd} &  (10 (x^4 - x^3 z + 2 x^2 y^2 + 3 x y^2 z + y^4) + z^4) (x^2 + y^2 - 
    z^2) + x^5 y \\
     (21)1 & {\rm nd} &  (10 (x^4 - x^3 z + 2 x^2 y^2 + 3 x y^2 z + y^4) + z^4) ((x + z)^2 + 
    y^2 - 2 z^2) + x^5 y  \\
       (11)2 & {\rm nd} &  (10 (x^4 - x^3 z + 2 x^2 y^2 + 3 x y^2 z + y^4) + 
    z^4) (x^2 + (y - z)^2 - z^2) + x^5 y  \\
     4  & {\rm nd} & x^6+y^6+z^6-4x^2y^2z^2 \\
     (41)  & {\rm nd} &  ((x^2 + 3 y^2 - 20 z^2) (4 x^2 + y^2 - 16 z^2) + 18 x^2 z^2) (x^2 + 
    y^2 - 10 z^2) - 2 z^6 \\
(41) & {\rm d} & \!\!\!\! 10 (((x^2 {+} 2 y^2 {-} 16 z^2) (2 x^2 {+} y^2 {-} 16 z^2) + x^2 y^2) (10 x +
       y + 5 z) + x z^4) (10 x - y - 8 z) - x z^5 \\
     (31)1 & {\rm nd} &  ((x^2 + 3 y^2 - 17 z^2) (3 x^2 + y^2 - 10 z^2) + 15 x^2z^2) (x^2 + 
    4 (y + z)^2 - 25 z^2) + x^3y^3 \\
     (21)2 & {\rm nd} &  ((x^2 {+} 3 y^2 {-} 20 z^2) (4 x^2 {+} y^2 {-} 16 z^2) + 
    18 x^2 z^2) ((x + y)^2 + 20 (x {-} y {-} 3 z)^2 - 24 z^2) + (y {-}     x) z^5\\
     (21)2 & {\rm d} &  ((x^2 + 3 y^2 - 20 z^2) (4 x^2 + y^2 - 16 z^2) + 18 x^2 z^2) (x^2 + 
    8 y^2 - 16 z^2) - 4 z^6 \\
     (11)3 & {\rm nd} &  ((x^2 + 2 y^2 - 30 z^2) (3 x^2 + y^2 - 20 z^2) + 
    15 x^2z^2) (x^2 + (4 y + 16 z)^2 - 15 z^2) + x^3y^3 \\
     5  & {\rm nd} &  4 ((x^2 + 2 y^2 - 4 z^2) (2 x^2 + y^2 - 4 z^2) + z^4) (x^2 + y^2 - 
    z^2) + x^3 y^3 \\
     (51) & {\rm nd} &  (  3 x^2 + 4 x y + 2 y^2-4 z^2) (x^2 + 2 (y - z)^2 - 8 z^2) (2 x^2 + 
    y^2 - 3 z^2) - z^6  \\ 
      (41)1 & {\rm nd} &  ( 4 x^2 + 6 x (y {-} z) + 3 (y {-} z)^2 -14 z^2) (x^2 + 
    5 (y {-} 2 z)^2 - 9 z^2) (2 x^2 + (y {-} z)^2 - 15 z^2) - y z^5 \\
        (31)2 & {\rm nd} &  ((x {+} z)^2 + 4 y^2 - 4 z^2) (7 (x {+} z)^2 +y^2 - 
    10 z^2) (  (x {+} z)^2 + 4 (2 (x {+} y) {+} 3 z)^2 -8 z^2) + x z^5  \\
     (21)3 & {\rm nd} & ((x {+} z)^2 + 3 y^2 - 4 z^2) (7 (x {+} z)^2 + y^2 - 
    12 z^2) (  (x {+} z)^2 + 3 (2 (x {+} y) {+} 3 z)^2-5 z^2) + x z^5   \\
    (11)4 & {\rm nd} & ((x^2 + 3 y^2 - 20 z^2) (4 x^2 + y^2 - 16 z^2) + 18 x^2 z^2) (8 x^2 + y^2 - 16 z^2) + (x {+} y) z^5   \\
     6  & {\rm nd} &  (3 x^2 + 5 x y + 2 y^2-7 z^2) (x^2 + 2 (y - z)^2 - 8 z^2) (2 x^2 + 
    y^2 - 5 z^2) - z^6  \\
     (61) & {\rm nd} & (4 x^2 + 4 x y + 3 y^2 - 4 z^2) (x^2 + 3 y^2 - 4 z^2) (4 x^2 + y^2 - 
    4 z^2) - z^6 \\
     (51)1  & {\rm nd} & 30 (((x {-} z)^2 + 3 y^2 - 5 z^2) (3 (x {-} z)^2 {+} y^2 {-} 5 z^2) + 
    x z^3) ((x {-} z)^2 + y^2 - 2 z^2) + (x {-} 2 z) z^5 \\
     (51)1  & {\rm d} & 7((x^2 + 3 (y + z)^2 - 48 z^2) (3 (x + z)^2 + y^2 - 48 z^2) - 
    z^4) (x^2 + y^2 - 26 z^2) + x z^5 + y z^5 \\
     (41)2  & {\rm nd} &  15 (4 x^2 + y^2 - 3 z^2) (x^2 + 3 y^2 - 
    3 z^2) ((4 x - z)^2 + 16 y^2 - 22 z^2) + (5 x z^5 + (y - z)^3 z^3) \\
     (31)3 & {\rm nd} &  34 ((3 x^2 + y^2 - 3 z^2) (x^2 + 8 y^2 - 3 z^2) + x^2 y^2) (2 x^2 - 
    y z - 2 z^2) +(x - 4 z) y z^4  \\
     (31)3 & {\rm d} &  ((x^2 + 3 y^2 - 28 z^2) (4 x^2 + y^2 - 20 z^2) - 
    z^4) ((x + z)^2 + y^2 - 12 z^2) - x z^5 \\
     (21)4 & {\rm nd} &  27 (2 x z - 6 y^2  + 2 y z+ 3 z^2 ) (- (x + y)^2 - 4 y^2  + 2 z^2 ) (5 (x + y)^2 + y^2-4 z^2) - x z^5  \\
     (11)5 & {\rm nd} & \! ((x^2 + 3 y^2 - 20 z^2) (4 x^2 + y^2 - 16 z^2) + 18 x^2 z^2) (16 x^2 + y^2 - 20 z^2) - (x + y) z^5 \\
     (11)5 & {\rm d} & ((x^2 {+} 3 y^2 {-} 20 z^2) (4 x^2 {+} y^2 {-} 16 z^2) + 
    18 x^2 z^2) ((x {+} y)^2 + 20 (x {-} y {-} 3 z)^2 - 24 z^2) + (x {+} y ) z^5\\
     7  & {\rm nd} & 2(4x^2 + y^2 - 4z^2)(x^2 + 4y^2 - 5z^2)(x^2 + y^2 - 4z^2) + 
 3x^4y^2 + x y^5  \\
\end{matrix}
\end{small}
$$

\noindent Next, we have eight topological types with eight ovals, all of which are non-dividing:
$$
\begin{small}
\begin{matrix}
     (71) & {\rm nd} & 2 (x^2 + y^2 - 26 z^2) ( x^2 + 3 (y + z)^2 - 48 z^2) (3 (x + z)^2 + y^2 - 48 z^2) - z^6 \\
          (21)5  & {\rm nd} &   40 (3 x^2 + y^2 - 3 z^2) (x^2 + 8 (y - z)^2 - 3 z^2) (2 x^2 - y z -     2 z^2) - (y^3 z^3 + 2 x z^5 - 2 z^6) \qquad \qquad  \\
       (11)6 & {\rm nd} &  19 (4 x^2 + y^2 - 4 z^2) (x^2 + 8 (y - z)^2 - 3 z^2) (2 x^2 - y z -     2 z^2) - (2 y - 3 z) z^5   \qquad \qquad \\
     8  & {\rm nd} & 12 (x^4 + 2 x^2 y^2 + y^4 - x^3 z + 3 x y^2 z) (7 (8 x + 3 z)^2 +  8 y^2 - 10 z^2) + x^5 y + 2 z^6  \qquad \qquad
\end{matrix}
\end{small}
$$
$$
\begin{tiny}
\begin{matrix}
     (61)1  & {\rm nd} & 
     (160075(5yz{-}x^2)(8(xz{+}15z^2)-(y{-}12z)^2)+109(17x{+}5y{+}72z)(
13x{+}5y{+}42z) (9x{+}5y{+}20z)(2x{+}5y))(5yz{-}x^2)-(x+3z)z^5 \\
      (51)2 & {\rm nd} & \!
      (5435525((y{+}z)z-x^2)((x{+}2z)z-2(y{-}x)^2)+5(25x{-}25y{-}31z)(5x{-}
50y{-}49z)(15x{+}25y{+}27z)(35x{+}25y{+}37z)) ((y{+}z)z-x^2)+x^5y \\
        (41)3 & {\rm nd} &  \!\!
       (14460138((y{+}z)z{-}x^2)((x{+}2z)z-2(y{-}x)^2)+5(25x{-}25y{-}31z)(
5x{-}50y{-}49z)(15x{+}25y{+}27z)(37z{+}35x{+}25y))  ((y{+}z)z{-}x^2)+x^5y \\
         (31)4 & {\rm nd} & 
         (27867506((y{+}z)z-x^2)((x{+}2z)z-2(y{-}2x)^2)+61(6x{+}8y{+}9z)(64y{+}63z
)(15x{-}25y{-}27z)(35x{-}25y{-}37z)) ((y{+}z)z- x^2)+x^5y \\
\end{matrix}
\end{tiny}
$$

\noindent
Among the $12$ rigid isotopy types with $9$ ovals, 
two  are from (\ref{eq:6aredividing}) and three pairs are from~(\ref{eq:8areboth}):
$$
\begin{small}
\begin{matrix}
        (11)7 & {\rm nd} &  23 (3 x^2 + y^2 - 3 z^2) (x^2 + 8 (y - z)^2 - 3 z^2) (2 x^2 - y z - 
    2 z^2) - (2 y - 3 z) z^5  \\
     9  & {\rm nd} & ((x^2 {+} 3 y^2 {-} 20 z^2) (4 x^2 {+} y^2 {-} 16 z^2) + 
    18 x^2 z^2) ((x {+} y)^2 + 20 (x - y - 3 z)^2 - 24 z^2) + y^2 z^4 \\
     9  & {\rm d} & ((x^2 + 3 y^2 - 20 z^2) (4 x^2 + y^2 - 16 z^2) + 18 x^2 z^2) (16 x^2 + y^2 - 20 z^2) + z^6  \\
          (81)  & {\rm d} &   ((x^2 + 3 y^2 - 28 z^2) (4 x^2 + y^2 - 20 z^2) - 
    z^4) (2 x^2 + y^2 - 12 z^2) - z^6 \\
 \end{matrix}
\end{small}
$$
$$
\begin{tiny}
\begin{matrix}
 (81) & {\rm nd} & (1920981  (y z {-} x^2 ) (57 (x {+} z) z - (6 x{-} y {+} 6 z)^2) +
   48  (10 x {+} 7 y {+} 3 z ) (11 x {+} 25 y {+} z) ( 11 x {-} 23 y{-}z) (
     10 x {-} 8 y{-}3 z )) (x^2 {-} y z)+ x^2 y^4 - 61  y^6   \\
(71)1 & {\rm nd} & (529321083 (y z {-} x^2) (53 (x {+} z) z-(6 x{-} y {+} 6 z)^2)+25 (10 x {+} 8 y {+} 3 z) 
(12 x {+} 30 y {+} z) (12 x {-} 32 y {-} z) (10 x {-} 8 y {-}3 z)) (x^2{-}y z)-y^6 \\
(61)2 & {\rm d} & (19157935 (5 y z{-}x^2) (8 (x z{+}15 z^2)-(y{-}12 z)^2)+1185 (17 x {+} 5 y {+} 72 z)
 (13 x {+} 5 y {+} 42 z) (9 x {+} 5 y {+} 20 z) (2 x{+}5 y)) (5 y z{-}x^2)-(x{+}3 z) z^5 \\
(51)3 & {\rm nd} & \!\!\! (28920269 ((y{+}z) z-x^2) ((x{+}2 z) z-2 (y{-}x)^2)+10 (25 x {-} 25 y {-} 31 z) (5 x {-} 50 y {-} 49 z) (15 x {+} 25 y {+} 27 z) (35 x {+} 25 y {+} 37 z))
((y{+}z) z{-}x^2)+x^5 y \\
(41)4 & {\rm nd} &
6761249083262 (68794627464 (1095368 (118 (x^2 + y^2 - 3 z^2) y 
+ (x {-} 2 z) (x {-} 12 z) (x - 13 z))    y + (x - 4 z) (x - 9 z) (x - 10 z) (x - 11 z)) y
\\ & & 
  + (x - 3 z) (x - 5 z) (x - 6 z) (x - 7 z) (x - 8 z)) y - z^6 \\     
(41)4 & {\rm d} & 13278270242890 (52982089012 (1610519 (149 (x^2+y^2-4 z^2) y
 +(x{-}3 z) (x {-}13 z) (x-14 z)) y+(x-5 z) (x-10 z) (x-11 z) (x-12 z)) y \\ & &
  +(x-4 z) (x-6 z) (x-7 z) (x-8 z) (x-9 z)) y-(x-5 z) z^5 \\
(31)5 & {\rm nd} & \! (26894836459 ((y{+}z) z{-}x^2) ((x{+}2 z) z-2 (y{-}2 x)^2)+1880 ( 6x {+} 8y {+} 9z) 
(64 y{+}63 z) ( 15x {-} 25y {-} 27z) (35 x {-} 25y {-} 37 z)) ((y{+}z) z-x^2)+x^5 y \\
(21)6 & {\rm d} &\! \!\!(93678589978 ((y{+}z) z{-}x^2) ((x{+}2 z) z-2 (y{-}2 x)^2)+50949 (6x {+} 8y {+} 9z) 
(18x {-} 72y {-} 73z) (5x {-} 6y {-} 7z) ({-}27 x{+}18 y{+}28 z)) ((y{+}z) z{-}x^2){+}x^5 y
 \end{matrix}
 \end{tiny}
  $$

\noindent The six topological types with ten ovals are all non-dividing:
$$
\begin{tiny}
\begin{matrix}
(91) & {\rm nd} & (40008  (y z - x^2 )  (57  (x + z)  z-(6 x- y + 6 z)^2)+( 10x+ 7y + 3 z)  (11 x+ 25y +  z
)  (11 x-23 y - z)  (10 x- 8y -3 z))  (x^2-y  z)-y^6 \\
(81)1 & {\rm nd} & (622771068  (y z {-} x^2 ) (57 (x {+} z) z - (6 x{-} y {+} 6 z)^2)+35  (10x {+} 8y {+} 3z)  (12x {+} 30y {+} z)  (12x {-} 32y {-} z)  (10x {-} 8y {-} 3z))  (x^2{-}y  z)-y^6 \\
(51)4 & {\rm nd} &
-3401397120  x^6-3195251840  x^5  y-2164525440  x^4  y^2-869728640  x^3  y^3+
332217600  x^2  y^4+316096000  x  y^5+53760001  y^6 \\ & &
\! +1597625920  x^5  z {+}36848468800000000000  x^4  y z  {+}
7988129600000000000  x^3  y^2  z 
{-} 3373286400000000000  x^2  y^3  z
{+} 3824761600000000000  x  y^4  z \\ & &  
+23425600000000000000000000000  y^4  z^2 
{-}1199390720  x^4  z^2 {-}7988129600000000000  x^3  y  z^2 
{-}127552392000000000000000000000  x^2  y^2  z^2 \\ & & 
+1618496000000000000  y^5  z
+764952320  x^3  z^3
 +\underbar{141724880000000000000000000000000000000}  y^3  z^3 
 +3654393600000000000  x^2  y  z^3   \\ & & 
-3824761600000000000  x  y  z^4  -130099200  x^2  z^4
+11712800000000000000000000000  y^2  z^4
+650496000000000000  y  z^5-2  z^6 \\ 
(41)5 & {\rm nd} &
\! -3401397120  x^6 {-}3195251840  x^5  y{-} 2164525440  x^4  y^2 {-} 869728640  x^3  y^3
{+}332217600  x^2  y^4 {+}316096000  x  y^5{+}53760002  y^6{+}1597625920  x^5  z \\ & & +36848468800000000000  x^4  y  z  +7988129600000000000  x^3  y^2  z
-3373286400000000000  x^2  y^3  z +3824761600000000000  x  y^4  z \\ & & 
+1618496000000000000  y^5  z 
 -1199390720  x^4  z^2  -7988129600000000000  x^3  y  z^2 
-127552392000000000000000000000  x^2  y^2  z^2 \\ &  &
+23425600000000000000000000000  y^4  z^2 
{+}764952320  x^3  z^3 
{+}3654393600000000000  x^2  y  z^3 
{-}3824761600000000000  x  y  z^4  {-}130099200  x^2  z^4  \\ & & 
+\underbar{141724880000000000000000000000000000000}  y^3  z^3 
+11712800000000000000000000000  y^2  z^4
+650496000000000000  y  z^5-z^6 \\
(11)8 & {\rm nd} & (227693  (yz-x^2)  ((x+2  z)  z-2  (y-2  z)^2)+(10x- 8y- 3 z)  
(10x- 23y- z)  (11 x+22 y+z)  (10 x+7 y+3z))  (x^2-y  z)+y^6 \\
10 & {\rm nd} &19  x^6-20  x^4  y^2-20  x^2  y^4+19  y^6
-20  x^4  z^2+60  x^2  y^2  z^2-20  y^4  z^2-20
  x^2  z^4-20  y^2  z^4+19  z^6 
\end{matrix}
     \end{tiny}
$$
Finally, here are representatives for the three  types with
the maximum number of ovals:
$$ \begin{tiny}
\begin{matrix} 
(91)1 & {\rm d} & (1941536164 (yz{-}x^2) (60 (x{+}z) z-(6 x{+}6 z{-}y)^2)+
118 (10x {+} 8y {+} 3z)(12x {+} 32y {+} z) (12x {-} 32y {-} z) (10x {-} 8y {-} 3z)) (x^2-y z)-y^6 \\
(51)5 & {\rm d} & -3401397120 x^6-3195251840 x^5 y-2164525440 x^4 y^2-869728640 x^3 y^3
+332217600 x^2 y^4+316096000 x y^5+53760001 y^6
\\ & & +1597625920 x^5 z+36848468800000000000 x^
4 y z+7988129600000000000 x^3 y^2 z-3373286400000000000 x^2 y^3 z
-130099200 x^2 z^4 
\\ & & \! -1199390720 x^4 z^2
{-}7988129600000000000 x^3 y z^2
{-}127552392000000000000000000000 x^2 y^2 z^2
{+} 23425600000000000000000000000 y^4 z^2 \\ & & +764952320 x^3 z^3
+3654393600000000000 x^2 y z^3 
+\underbar{141724880000000000000000000000000000000} y^3 z^3 
+3824761600000000000 x y^4 z
\\ & & +1618496000000000000 y^5 z
-3824761600000000000 x y z^4
+11712800000000000000000000000 y^2 z^4
+650496000000000000 y z^5-z^6 \\
(11)9 & {\rm d} & (340291 (yz-x^2)  ((x+2  z)  z-2  (y-2  z)^2)+(10x- 8y- 3 z) (12x- 27y- z)
 (12 x+28 y+z) (10 x+7 y+3 z)) (x^2-y z)+y^6 
\end{matrix}
\end{tiny}
$$

The correctness of the $64$ polynomials above
was verified using the tools to be described in Section 3.
In Section 4 we present an algorithm for sampling
from the connected components of
$\PP^{27}_\RR \backslash \Delta$.
The underlined coefficient is closest to the bound $1.5 \times 10^{38}$
in Proposition \ref{prop:bigintegers}.

To construct our list of $64$ representatives, we relied
on the following three methodologies. First, we employed
geometric constructions that are described in the
literature on topology of real algebraic curves.
Such constructions are found in many of the articles,
from Harnack and Hilbert to Gudkov and Viro. 
For larger numbers of ovals, these constructions led
to polynomials whose integer coefficients were too big.
In those cases, we needed to improve the coefficients.
This was done using techniques described in Section 4. A third approach
is to start with especially nice, but possibly singular, 
ternary sextics seen in the literature.

The last method is exemplified by the  {\em Robinson curve}.
This is the symmetric sextic
\begin{equation}
\label{eq:robinson}
 \mathcal{R}(a,b,c) \,=\,\,
 a (x^6+y^6+z^6) \,+\, b x^2 y^2 z^2 \,+ \,c (x^4 y^2 + x^4 z^2 + x^2 y^4 + x^2 z^4 + y^4 z^2 + y^2 z^4),  \quad
 \end{equation}
 where $(a:b:c) $ is any point in $\PP^2_\RR$.
For $(a:b:c) = (19:60:-20)$, this degree six curve is smooth 
and its real locus consists of ten non-nested ovals.
This is a sextic we like, and  we therefore included $\mathcal{R}(19,60,-20)$ 
in the list above as our
representative for the Type 10 nd.

\begin{figure}[h]
\begin{center}
  \includegraphics[height=6.7cm]{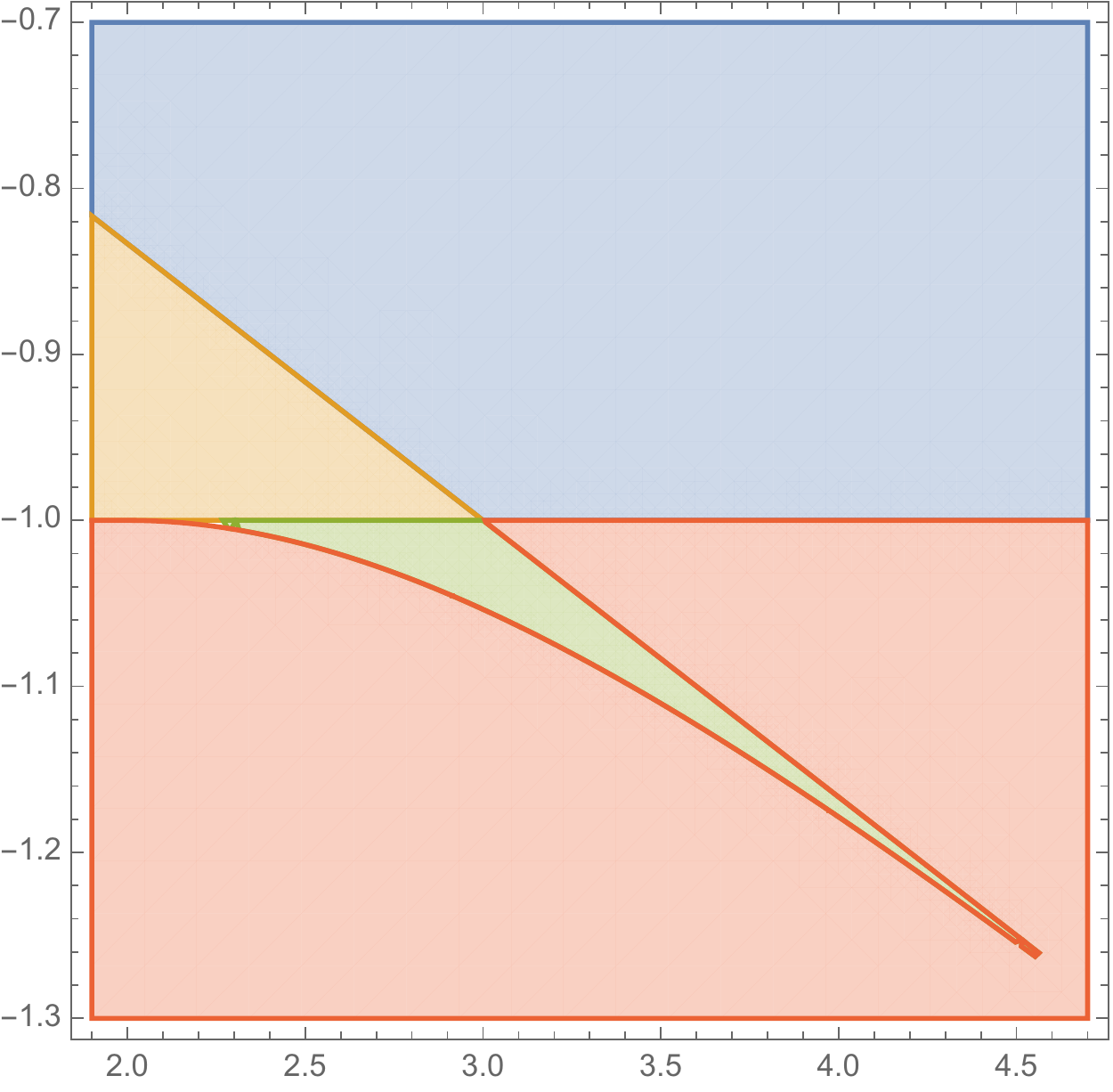}
  \end{center}
\vspace{-0.1in}
    \caption{\label{fig:zwei} 
The discriminant divides the Robinson net (\ref{eq:robinson}) into $15$ components that realize
 four topological types.
The green region represents smooth sextics with $10$ non-nested ovals.}
         \end{figure}

We named this net of sextics after R.M.~Robinson, who showed in 1973
that the special sextic $\mathcal{R}(1,3,-1)$ is non-negative but is not
 a sum of squares. 
The reason is geometric: its real locus consists of  $10 $ isolated 
singular points in $\PP^2_\RR$, given by the columns of the matrix
\begin{equation}
\label{eq:robinson10}
\begin{pmatrix}
               \,  1  &  -1 &    1  &   1  &  0  &   0  &  1  &  -1  &  1  &   1 \\
                \, 1  &   1 &   -1  &   1  &  1  &   1  &  0   &  0  &  1  &  -1 \\
               \,  1  &   1  &   1  &  -1  &  1  &  -1  &  1  &   1 &   0  &   0 
\end{pmatrix}.
\end{equation}
To understand how the topology of $\mathcal{R}(a,b,c)$
varies with $(a:b:c)$,
we examine the complement of the discriminant  in $\PP^2_\RR$.
This is the following reducible polynomial of degree~$75$:
\begin{equation}
\label{eq:robinsondisc}
a^3 (a+c)^6(3a-c)^{18}(3a+b+6c)^4 (3a+b-3c)^8 (9 a^3-3 a^2b+ab^2-3ac^2-bc^2+2c^3)^{12}
\end{equation}
The complement of the discriminant in $\PP^2_\RR$ has 15 connected components.
These realize the following topological types of smooth sextics: 10, 4, 3 and 0.
 These types are found in 1, 3, 5 and 6
connected components respectively. The most interesting part
of the partition is shown in Figure \ref{fig:zwei}. The green region is the component corresponding to curves with 10 ovals. Smooth curves in the orange, red and blue region have 0, 3 and 4 ovals
respectively.

The Robinson net serves as 
a cartoon for the geometric object studied in this paper:
the partition of a $27$-dimensional real projective space  into $64$ connected
components by an irreducible hypersurface of degree $75$.
Figure \ref{fig:zwei} is a two-dimensional slice of that partition.

\section{Classifier and  Empirical Distributions}

Deriving the topology of a real plane curve from its equation
is a well-studied problem in computational geometry. The main idea is to
ascertain the topology of plane curves by constructing isotopic graphs whose nodes
are the critical points (singular and extreme).
Authors  who  studied  this  problem include Jinsan Cheng, Sylvain Lazard, Luis Pe\~{n}aranda, Marc Pouget, Fabrice Rouillier, Elias Tsigaridas,
Laureano Gonz\'alez-Vega, Ioana Necula, M'hammed El Kahoui, Hoon Hong, Raimund Seidel, and Nicola Wolpert. See \cite{CL, EKW, GN, SW}.

The algorithm in \cite{CL} is called {\tt ISOTOP}.
It is implemented in {\tt Maple} using certain packages written in {\tt C}. We
refer to \cite[Table 1, page 28]{CL} for comparisons between {\tt ISOTOP}
and {\tt Top}, as well as with {\tt INSULATE,} {\tt AlciX} \cite{EKW}, and {\tt Cad2d} \cite{CWB}.
The latter two are implemented in {\tt C}++.

After first experiments with these packages, we came to the conclusion
that it is preferable for us to have our own specialized implementation for
curves of degree six in $\PP^2_\RR$. In particular, all of the
mentioned packages compute the topology of the curve in the affine
chart $\{z=1\}$. It was not obvious how to extract the topological
type of the curve in the projective plane from the output with a reasonable amount of coding effort.
Developers and users of the packages above may wish to experiment 
with our representatives for $10$ and $11$ ovals.

We wrote a program in {\tt Mathematica} called {\tt
  SexticClassifier}. It relies heavily on the built-in quantifier
elimination techniques of {\tt Mathematica}. The code can be obtained
from our supplementary materials website (\ref{eq:url}).  The input to
{\tt SexticClassifier} is a ternary sextic with integer coefficients,
$f \in \ZZ[x,y,z]_6$. The code checks whether $f$ defines a
non-singular curve in $\PP^2_\CC$.  If not, then the output is
``singular''. Otherwise, our program identifies which of the $56$
topological types the real curve $V_\RR(f)$ belongs to.
We next explain how it works.

First we compute a \textit{Cylindrical Algebraic Decomposition} (CAD; see
e.g.~\cite{BCR, CWB}) of the curve $V_\RR(f)$ in the affine chart
$z=1$. From this we build a graph whose nodes are the critical points of the
projection along the $y$-axis. Two nodes are connected
by an edge if an arc of the curve connects the corresponding points.
We also keep track of the
relative positions of these arcs. In order to get the correct topology in $\PP^2_\RR$,
we add further edges corresponding to
arcs crossing the line at infinity. We end up with a graph whose
connected components are in one-to-one correspondence with the connected
components of the curve. Finally, for each pair of connected
components of the graph, we have to check whether one of the
corresponding ovals lies in the inside of the other. This is done by first
deciding whether the center of projection lies inside the ovals or
not. Knowing this, detecting a nesting of two ovals only amounts to
looking at the parity of the number of branches of one oval lying
above the other oval.
The program terminates correctly in less than four seconds for all
sextics in Section 2. If the coefficients are between $-10^8$ and $10^8$,
then it takes less than one second.
\begin{example} 
\label{ex:emptyavoidancelocus} \rm
Here is an instance that will be of interest in Section 5. Let $f$ be the  sextic
$$ 
7(x+y+2z)(x+2y+z)(2x+y+z)(x-2y+3z)(y-2z+3x)(z-2x+3y)+ xyz(x^3+y^3+z^3).
$$
After checking that the complex curve $V_\CC(f)$ is smooth, our code reveals
that the real curve $V_\RR(f)$  consists of three separate non-nested ovals.
Thus, the input to  {\tt SexticClassifier} is the polynomial $f$
and the output is the label 3. In particular, $V_\RR(f)$ is non-dividing.

The sextic $f$ has the property that every real line in $\PP^2$
meets $V_\CC(f)$ in at least one real point. Thus, $V_\RR(f)$ is not
compact in any affine chart of $\PP^2_\RR$, regardless of which line
serves as the line at infinity. While such non-compact curves may cause difficulties in some of
 the approaches discussed above, {\tt SexticClassifier} has been designed
to handle them well.
\hfill $\diamondsuit$
\end{example}

Forty-eight of the fifty-six topological types determine whether the
curve is dividing or not. However, the remaining eight were found by
Nikulin \cite{Nik} to split into two rigid isotopy types. Suppose that
the output of {\tt SexticClassifier} is one of the eight labels
in~(\ref{eq:8areboth}).  At present, we have no easy tool for deciding
whether the given $f$ is dividing or non-dividing. That decision
requires us to build a model of the Riemann surface $V_\CC(f)$ in
order to ascertain whether $V_\CC(f) \backslash V_\RR(f)$ has one or
two connected components.  A method for making that decision was
developed and implemented by Kalla and Klein \cite{kalla}, but
presently their code is not suitable for curves of genus $10$. One of
our goals for the future is to extend {\tt SexticClassifier} so that
it decides rapidly between `d' and `nd' when the output is in
(\ref{eq:8areboth}).

As a first application of {\tt SexticClassifier}, we 
computed the empirical distributions of the topological types
over the space of ternary sextics.
In other words, we ask: what is the probability that a particular
topological type arises when we pick a sextic curve at random?
Of course, the answer will depend
on the choice of a probability distribution on the sextics
\begin{equation}
\label{eq:sextic}
 f \quad = \quad \sum_{i+j+k=6} c_{ijk} x^i y^j z^k .
\end{equation}
A theoretical study for curves of large degree was carried out recently by
Lerario and Lundberg \cite{LeLu}, who employed the
real Fubini-Study ensemble and the Kostlan distribution.
Our experiments below are meant to inform this line of inquiry
with some empirical numbers.

The first distribution we consider is ${\rm U}(3)$-invariant.
 The $28$ coefficients $c_{ijk}$ 
 are~chosen independently from a univariate normal distribution, centered at $0$,
with variance equal to the multinomial coefficient $6!/(i! j! k!)$.
According to   \cite[\S 16.1]{BC}, this is  the unique
${\rm U}(3)$-invariant probability measure on $\RR[x,y,z]_6$.
We selected $1,500,000$  samples, we ran {\tt SexticClassifer} on these sextics, 
 and we tallied the topological types.
 The result is shown in Table~\ref{tab:U3invariant}.
 
\begin{table}[b]
 \setcounter{MaxMatrixCols}{20}
 $$\boxed{ \begin{small} \setlength{\arraycolsep}{4pt}\begin{matrix}
 1 & 2 & 3 & (11) & 4 & (11)1 & (21) & 5 & \emptyset & \! (11)2 & \! (21)1 & 6 & (31) & \! ({\rm hyp}) \\
 875109 & 423099 & 97834 & 90316 & 7594 & 4360 & 1180 & 245 & 127 & 118 & 8 & 7 & 2 & 1 
\end{matrix}
  \end{small}}
 $$
 \vspace{-0.22in}
 \caption{\label{tab:U3invariant}
 Counts of topological types sampled from the $U(3)$-invariant distribution}
\end{table}
 
 We see that the empirical distribution is very skewed. Only $14$ of the
 $56$ types were observed at all. Only six types had an empirical
 probability of $ \geq 1\%$. No curve with more than six
 ovals was observed. 
 In our sample, the average number of connected components is
 approximately $1.50$. Another numerical invariant of the topological
 type of a  smooth real plane curve introduced in \cite[page 8]{LeLu} is the \textit{energy}. This nonnegative integer measures the nesting of the ovals. For sextics, the maximal energy is $38$ and attained by 
 the Harnack-type curve $(91)1$. The average energy of 
 the sextics in the sample above is approximately $2.99$.

We experimented with several other distributions on $\RR[x,y,z]_6$, each time drawing 
500,000 samples and running {\tt SexticClassifier}. In the
following tables, we report percentages, and we only list topological types
with empirical probability at least 0.01\%.

 \begin{table}[!h]
 \setcounter{MaxMatrixCols}{20}
 $$ \boxed{\begin{small} \begin{matrix}
 1       & 2 & 3 & (11) & \emptyset & 4\\
 77.52\% &  18.19\% & 2.11\% & 1.46\% & 0.66\% & 0.06\%
  \end{matrix}
  \end{small}}
 $$
 \caption{\label{tab:uniform}
 Sextics with coefficients in $\{-10^{12},\ldots,10^{12}\}$ uniformly distributed}
 \end{table}

 \begin{table}[!h]
 \setcounter{MaxMatrixCols}{20}
 $$ \boxed{\begin{small} \begin{matrix}
 1       & \emptyset & 3       & (11)   & 4      & 6      & (31)   & (11)3  & 7 & ({\rm hyp}) \\
 45.69\% & 28.38\%   & 16.15\% & 7.40\% & 2.17\% & 0.13\% & 0.03\% & 0.03\% &	0.01\% &  0.01\%
  \end{matrix}
  \end{small}}
 $$
 \caption{\label{tab:symmetric}
 Symmetric sextics with coefficients in $\{-10^{12},\ldots,10^{12}\}$ uniformly distributed}
  \end{table}

   \begin{table}[!h]
 \setcounter{MaxMatrixCols}{21}
 $$ \boxed{\begin{small} \begin{matrix}
    2 & 3 & 1&4&\!(11)&\!(11)1&5&(21)&\!\! (11)2&6 & \! (11)3&\!(21)1&(31)&7\\
\!    29.12\% \! & \! \! \! 25.77\%  \!\! & \!\!\! 16.44\% \!\! & \!\! \!11.06\% \! 
\! & \!\!\! 8.02\% \!\! &\! \! \! 4.30\%\! \!  & \!\! \!2.46\% \!\! &\! \! \!1.19\% \!\! &\! \!\!0.98\%\!\! &
\!\! \! 0.30\% \! \!    &\! \!\!  0.13\% \!\!&\! \!\! 0.12\% \!\! &\! \!  \! 0.07\% \!\! & \!\! \! 0.02\% \!\!
  \end{matrix}
  \end{small}}
 $$
 \caption{\label{tab:detrep}
 Sextics that are determinants of random symmetric matrices with linear entries}
 \end{table}
 
 \begin{table}[!h]
 \setcounter{MaxMatrixCols}{20}
 $$ \boxed{\begin{small} \begin{matrix}
 n&1   & \emptyset  & 2 & (11) & 3 \\
 10&90.17\% & 5.14\% & 4.50\% & 0.09\% & 0.09\%  
 \\11&89.95\% & 4.75\% & 5.06\% & 0.12\% & 0.12\%   
 \\12& 89.90\% & 4.28\% & 5.53\% & 0.15\% & 0.15\%    
  \end{matrix}
  \end{small}}
 $$
 \caption{\label{tab:rank10}
Sextics that are signed sums of $n$ sixth powers of linear forms}
 \end{table}

If we naively sample our sextics having uniformly distributed integer coefficients, then
the empirical distribution is even more skewed (Table \ref{tab:uniform}). In Table \ref{tab:symmetric} we see the empirical distribution obtained from sampling symmetric sextic polynomials. 
We do this by taking linear combinations of the monomial symmetric polynomials with coefficients being uniformly distributed integers between $-10^{12}$ and $10^{12}$. We see that the distribution now looks rather different from the two distributions considered before. For example, Type 7 
appears with probability $0.01\%$ whereas it did not appear among the $1,500,000$ samples from the $U(3)$-invariant distribution. The largest variety of types is observed when we sample sextics of the form $\det(xA+yB+zC)$ where $A,B$ and $C$ 
are symmetric $6 \times 6$-matrices whose entries are
   uniformly distributed random integers between $-1000$ and $1000$ (Table \ref{tab:detrep}). This is also the only distribution we considered where the type with only one oval is not the most common type. Several types appear that did not show up among the $1,500,000$ 
    samples in Table  \ref{tab:U3invariant}.
    
   Our most skewed distribution was from
   sampling signed sums (with the signs chosen uniformly at random) of ten sixth powers of linear forms whose coefficients are uniformly distributed integers between $-1000$ and $1000$.
   Thus, here we are restricting to sextics of real rank $10$,
   the case considered in \cite[\S 6]{MMSV}.
   Table \ref{tab:rank10} reveals that more than $90\%$ of the samples
   have one oval. After passing to eleven summands, we observe more curves of Type $2$ than empty curves. Going to sums of twelve sixth powers of linear forms
   increases this effect.

Our experiments demonstrate that it is extremely rare to observe
many ovals when sextic curves are generated at random. 
We never encountered a sextic with  8, 9, 10 or 11  ovals.
Only few types occurred in our samples.
This underscores the importance of having the
explicit polynomials in $\ZZ[x,y,z]_6$ that are listed above,
to serve as seeds for local sampling.

\section{Constructing Representatives}\label{sec:construct}

In this section we discuss the construction of representative sextics,
such as those listed in Section~2. We also present a method for
sampling from any of the $64$ rigid isotopy classes.

Small coefficient size is  a natural criterion for desirable representatives.
We say that a sextic $f$ as in (\ref{eq:sextic}) is {\em optimal}  if its coefficients $c_{ijk}$
are integers, its complex curve $V_\CC(f)$ is smooth,
and the  largest absolute value $|c_{ijk}|$
is minimal among all such
 sextics in the same rigid isotopy class. For instance, the
Fermat sextics $x^6+y^6 \pm z^6$  are optimal. One approach to finding optimal sextics
 is to sample at random from sextics with $|c_{ijk}| \in \{0,1,\ldots,m\}$ with $m$ very small.
 One might also do a brute force search that 
 progressively increases the sum of the absolute values $|c_{ijk}|$.
Such strategies work for some of the types seen with highest frequency in
Table \ref{tab:U3invariant}. For instance, sampling with $m=1$
yields these four optimal sextics:
$$ 
\begin{small}
\begin{matrix}
2 & {\rm nd} &
x^6-x^5 y-x^5 z-x^4 y z+x^3 y^3+x^3 y z^2-x^2 y^4  -x^2 y^2 z^2+x y^4 z \\ & &
-x y^3 z^2-x y^2 z^3-x y z^4+y^6+y^5 z+y^4 z^2+y^3 z^3-y z^5+z^6 \\
3 & {\rm nd} &
x^6-x^5 y-x^4 y^2+x^4 z^2+x^3 y^3+x^3 y^2 z+x^3 y z^2+x^2 y^3 z-x^2 y^2 z^2
+x^2 y z^3+x^2 z^4 \\ & & \quad  +x y^4 z+x y^3 z^2+x y^2 z^3+x y z^4-x z^5+y^6+y^5 z
+y^4 z^2-y^2 z^4-y z^5+z^6 \\
(11) & {\rm nd}  &
x^5 y+x^5 z+x^4 y^2+x^4 y z-x^3 y^3+x^3 y z^2-x^3 z^3-x^2 y^4+x^2 y^3 z
+x^2 y^2 z^2-x^2 y z^3 \\ & & \quad +x^2 z^4-x y^4 z+x y^3 z^2-x y z^4+x z^5-y^5 z
-y^4 z^2+y^3 z^3+y^2 z^4-y z^5-z^6 \\4 & {\rm nd} & -x^6 + x^5 y + x^4 y^2 - x^3 y^3 + x^2 y^4 + x y^5 - y^6 + x^5 z +  x^4 y z + x y^4 z \\ & & \quad + y^5 z + x^4 z^2 + y^4 z^2 - x^3 z^3 - y^3 z^3 +  x^2 z^4 + x y z^4 + y^2 z^4 + x z^5 + y z^5 - z^6
\end{matrix}
\end{small}
$$
However, this approach is not useful for constructing the vast majority of types,
since these are extremely rare when we sample at random, and they will never appear
for small $m$.

For all systematic constructions, it is imperative to work with the
{\em discriminant} $\Delta$.
Each connected component of the complement $\RR[x,y,z]_6 \backslash \Delta$
corresponds to one of our $64$ types.
We identify $\Delta$ with its defining irreducible polynomial  
over $\mathbb{Z}$ in the $28$ unknowns $c_{ijk}$.

We evaluate $\Delta$ using
{\em Sylvester's formula}, as stated by Gelfand, Kapranov and Zelevinsky
 \cite[Theorem 4.10, Chapter 3]{GKZ}. This expresses
 $\Delta$  as the determinant of a $45 \times 45$-matrix $\mathcal{S}_f$.
 Each entry in the first $30$ columns of $\mathcal{S}_f$ is
 either $0$ or one of the coefficients $c_{ijk}$.
 The entries in the last $15$ columns are cubics
 in the $c_{ijk}$. So, the degree of ${\rm det}(\mathcal{S}_f)$ is $75$,
 as required. The \emph{Sylvester matrix} $\mathcal{S}_f$ is the representation 
 in monomial bases of an $\RR$-linear map
 $$ \mathcal{S}_f \,:\,(\,\RR[x,y,z]_3\,)^3 \,\oplus \,\RR[x,y,z]_4
\,\,\longrightarrow\,\, \RR[x,y,z]_8 $$
that is defined as follows. 
On the first summand, it maps a triple of cubics to an octic via
$$
\mathcal{S}_f \,:\,(a,b,c)\, \mapsto \,
a \frac{\partial{f}}{\partial{x}} + 
b \frac{\partial{f}}{\partial{y}} +
c \frac{\partial{f}}{\partial{z}} .
$$
On the second summand, 
the map $\mathcal{S}_f$ takes a quartic monomial 
$ x^r y^s z^t$ to the octic ${\rm det}(M_{rst})$,
where $M_{rst}$ is any $3 \times 3$-matrix
of ternary forms that satisfies the homogeneous identity
$$ \begin{pmatrix} 
\partial{f}/\partial{x} \\
\partial{f}/\partial{y} \\
\partial{f}/\partial{z} 
\end{pmatrix}\,\, = \,\, M_{rst} \cdot
\begin{pmatrix}
x^{r+1} \\
y^{s+1} \\
z^{t+1} \\
\end{pmatrix}.
$$
The entries of $M_{rst}$ are linear in the
$c_{ijk}$, so ${\rm det}(M_{rst})$ is
an octic in $x,y,z$ whose coefficients are
cubics in the $c_{ijk}$.
These  are the entries in the column
of $\mathcal{S}_f$ that is indexed by $x^r y^s z^t$.

\begin{proposition} \label{prop:sylvester}
The discriminant $\Delta$ equals the determinant of the $45 \times
45$-matrix $\mathcal{S}_f$.
\end{proposition}

\begin{proof}
We use Sylvester's formula for
the resultant of three ternary quintics. This is 
 \cite[Theorem III.4.10]{GKZ}
for $d=5$ and $k=4$. If we take the three quintics
to be the three partial derivatives of $f$, then we get
the matrix $\mathcal{S}_f$ above. That resultant
equals our discriminant because both are non-zero 
homogeneous polynomials
of the same degree $75$ in the $c_{ijk}$.
 \end{proof}

A pencil of sextics is a line  $\{f + t g\}$ in the space
$\PP^{27}$ of all sextics.
Its discriminant $\Delta(f + t g)$ is a univariate
polynomial in $t$ of degree $75$.
 We can compute that polynomial as the determinant of the
 Sylvester matrix $\mathcal{S}_{f + tg}$. For two
sextics $f$ and $g$ in $\ZZ[x,y,z]_6$ with reasonable coefficients, we obtain
 the discriminant $\Delta(f+tg)$ in a few seconds.
This method works, in principle, also for evaluating $\Delta$
on families with more than one parameter. For instance,
we get the output (\ref{eq:robinsondisc})
from the input (\ref{eq:robinson}) in under one second.
However, that output factors and is small.
In our experience, the symbolic evaluation of the $45 \times 45$ determinant
 in Proposition \ref{prop:sylvester}
works well for pencils of sextics, but generally fails for nets of sextics.

Sylvester's formula allows us to sample from a fixed rigid isotopy
class.  Namely, we start with a representative $f$ with
$\Delta(f) \not= 0 $, like one of the $64$ sextics  in
Section~2.  We then pick a random sextic $g$ and we compute the
univariate polynomial $\Delta(f + t g)$.  This has $75$ complex
roots. We extract the real roots, and we identify the largest negative
root and the smallest positive root. For any $t$ in the open interval
between these two roots, the sextic $f + tg$ has the same rigid
isotopy type as $f$.  Repeating this many times, we thus sample from
the connected component of $\RR[x,y,z]_6 \backslash \Delta$ that
contains $f$. This gives us access to all sextics in
the largest star domain with center $f$ contained in that component.
We call this process the {\em local exploration method}. It will be
used for the applications in Sections 5 and 6.

For exploring the $64$ connected components 
of $\PP^{27}_\RR \backslash \Delta$, it is
important to understand their adjacencies.
A general point in the discriminant $ \Delta$ is a sextic curve $f$ that has precisely one
ordinary node. If $f$ is in the real locus
$\Delta_\RR$, then that node is a point in the real plane $\PP^2_\RR$.
Two of the $64$ types are connected by a \textit{discriminantal transition} if there is a curve in the closure of both of the components having only one singular point which is an ordinary node.

 There are three different types of discriminantal transitions.
 If the singular curve has an isolated real point (acnode), defined locally by 
 $x^2+y^2=0$, then the transition corresponds to removing one of the empty ovals.
  We call this operation {\em shrinking of ovals}.
Itenberg \cite{itenberg} uses the term {\em contraction}.
  The inverse operation is adding an empty oval.
 
 \medskip
 The following lemma and the subsequent theorem are probably well-known to experts in the area. We include the proofs for lack of a suitable reference.
 
 \begin{lemma}
 Shrinking an oval always leads to a curve of non-dividing type.
\end{lemma}

\begin{proof}
Consider a real plane curve $C$ with only one singularity $p$ that is an acnode. Since being of non-dividing type is an open condition, we can assume that $C$ is of dividing type. Then $p$ is in the closure of both connected components of $C_\CC \backslash  C_\RR$. In particular, 
$(C_\CC \backslash C_\RR)\cup \{p\}$ is connected. Therefore, after shrinking 
an oval we get a curve of non-dividing type.
\end{proof}
 
The other type of ordinary node consists of two crossing real branches
(crunode), defined locally by $x^2-y^2 = 0$.  There are two
possibilities: The connected component containing the node is either two ovals intersecting in one point, or two
pseudolines intersecting in one point. The case of one
oval and one pseudoline cannot occur in even degree. In the following we describe the topology of small perturbations of the nodal curve in each of these two cases. In the former
case, the transition consists of two ovals coming together and forming
one oval. This happens in one connected component of the complement of
all other ovals.  We call this transition {\em fusing of ovals}.
Itenberg \cite{itenberg} uses the term {\em conjunction}. This
operation reduces the number of ovals by one.  When two
pseudolines intersect, every small perturbation of the nodal curve has the same number of ovals but the interior and exterior of one outermost oval are exchanged. We call
this operation {\em turning inside out}. An example of turning inside
out is shown in Figure \ref{fig:turning}.  For plane conics, turning
inside out is the only possibility.  For quartics, all three
transitions are possible.  We summarize our discussion as follows.

\begin{theorem}
For curves of even degree, every discriminantal transition
between rigid isotopy types is one of the following:
shrinking of ovals, fusing of ovals, and turning inside~out.
\end{theorem}

\begin{proof}
Let $C$ be a real plane curve of even degree with exactly one ordinary singularity $p$. If $p$ is an acnode, then $C$ corresponds to shrinking. Let $p$ be a crunode. There are two subsets $C_1, C_2 \subset C_\RR$,
 both homeomorphic to the circle, such that $C_1\cap C_2=\{p\}$.
Let $\pi: \tilde{C}\to C$ be the normalization map. The
fiber $\pi^{-1}(p)$ consists of exactly two points $p_1,p_2\in\tilde{C}_\RR$.

Suppose that $p_1$ and $p_2$ belong to the same connected component of $\tilde{C}_\RR$.
If $C_1$ or $C_2$ does not disconnect $\PP^2_\RR$, there would be a small deformation of $C$ to a smooth curve having (at least) one pseudoline as one of the connected components of its real part. Since this is not possible, both $C_1$ and $C_2$ disconnect $\PP^2_\RR$. 
This  case corresponds to fusing of ovals.

Next suppose that $p_1$ and $p_2$ belong to different connected components of $\tilde{C}_\RR$. For both bifurcations of the node, the number of connected components of the real part of the curve stays the same. If $C_1$ or $C_2$ disconnected $\PP^2_\RR$, then there would be another intersection point of $C_1$ and $C_2$ besides $p$. 
Thus, $\PP^2_\RR \backslash C_i$ is connected for $i=1,2$,
and $\PP^2_\RR  \backslash (C_1\cup C_2)$ has two connected components, both homeomorphic to an open disc. Depending on the bifurcation of the node, one of these connected components is still homeomorphic to an open disc after deformation and the other one 
is not. This corresponds to turning inside out.
\end{proof}

\begin{figure}[h]
\begin{center}
\vspace{-0.2in}
  \includegraphics[width=7.5cm]{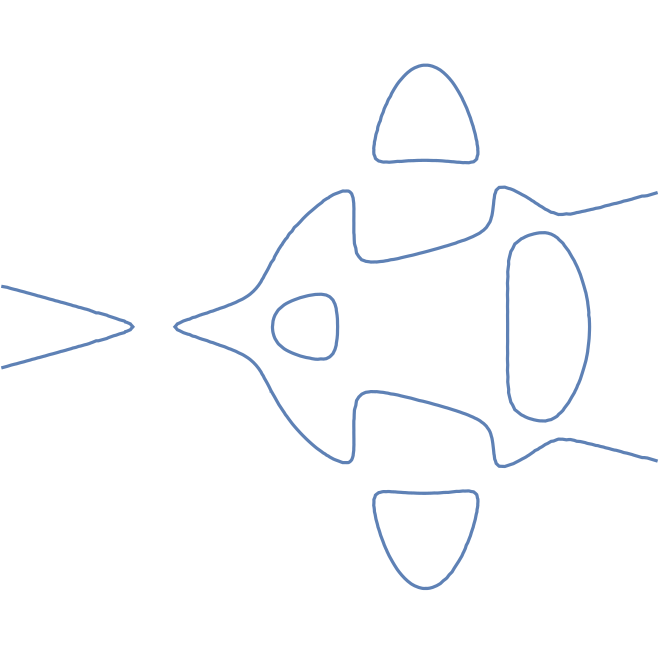}\hspace*{2.3em}
  \includegraphics[width=7.5cm]{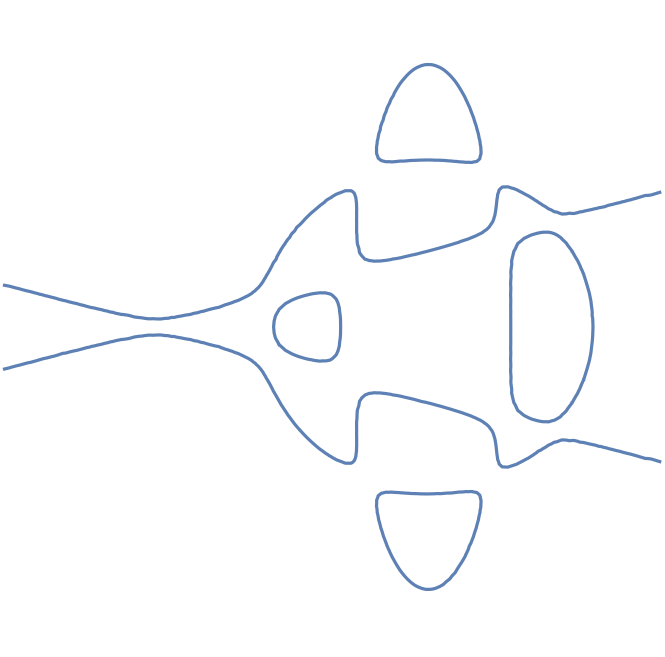}
\vspace{-0.3in}
\end{center}
    \caption{\label{fig:turning} 
Type (21)2d transitions into Type (21)2nd by turning an oval inside out.}
         \end{figure}

It is instructive to examine the diagram in Figure \ref{fig:eins} from the
perspective of discriminantal transitions. The edges in the poset
correspond to shrinking or fusing. There are three possibilities for
what might be geometrically possible: shrinking only, fusing only, or
shrinking and fusing. For instance, Type $(11)$ can become 
Type $1$ by either shrinking the inner oval, or by fusing the two
nested ovals. Both possibilities are geometrically realized by a
singular curve with a single node that lies in the common boundary
between the two types.

\begin{theorem}[Itenberg]
\label{thm:itenberg}
Each of the edges in Figure \ref{fig:eins} is realized by shrinking an empty oval, except the one between (hyp) and (11). Not every edge is realized by fusing two ovals.
\end{theorem}

\begin{proof}
The first statement is \cite[Prop.~2.1]{itenberg}. Furthermore, it was shown in \cite{itenberg} that the transition from $(11)9$ to $10$ cannot be realized by fusing.
\end{proof}

One possible way of explicitly realizing edges by fusing is to use Gudkov's constructions \cite{Gud} and the following lemma which is a special case of a theorem due to Brusotti \cite{Brus}.

\begin{lemma}
 Let $C_1, C_2 \subset \PP^2$ be two smooth real curves of degrees $2$ and $4$ (resp.~$1$ and $5$) intersecting transversally. 
 By a small perturbation, we can fix any one of the real nodes of the sextic curve $C_1\cup C_2$ and perturb all the others independently in any prescribed manner.
\end{lemma}

\begin{proof}
 Let $q,p_1,\ldots,p_{7}\in\PP^2_\RR$ be eight distinct real points lying on the smooth quadric $C_1$. We claim that,
 for every tuple $\epsilon\in\{\pm 1\}^{7}$,
 there is a sextic which is singular at $q$ and whose sign at $p_i$ is $\epsilon_i$.
 Let $L$ be the linear system of all sextic curves that are singular at $q$. The pull-back of $L$ to $C_1\cong\PP^1$ is the set of all bivariate forms of degree $12$ having a double root at $q$. Since for any distinct 7 points in $\RR$ there is a polynomial of degree $10$ that vanishes on all but one of these points, the claim follows.
 The other case (degrees $1$ and $5$) is analogous.
 \end{proof}

We might also approach these questions
 by computational means. This requires  a software tool for
 the following task. Consider two general sextics
$f,g \in \ZZ[x,y,z]_6$ and compute the univariate polynomial
$\Delta(f+tg)$ of degree $75$. For each of its real roots $t^*$,
we must decide if the transition at $t^*$ is a
shrinking of ovals, a fusing of ovals, or turning inside out.

Let us now examine our third discriminantal transition.
Turning inside out preserves the number of ovals, so it
is an operation that acts on each of the rows in Figure~\ref{fig:eins} separately.

\begin{proposition} \label{prop:turningout}
If we turn an outermost oval of a smooth sextic inside out, then the topological type of the resulting curve is the one obtained by reflecting Figure \ref{fig:eins} vertically.
If the curve was dividing, then it is non-dividing after turning inside out. Curves that are non-dividing can become either dividing or non-dividing. 
\end{proposition}

\begin{proof}
Let $C$ be a real plane curve with exactly one ordinary crunode $p$. In a neighborhood of $p$, the Riemann surface $C_\CC$ is homeomorphic to the union of two discs $D_1$ and $D_2$ with $D_1\cap D_2=\{p\}$. The real part $C_\RR$ divides $D_1$ and $D_2$ into two connected components $D_1^+$, $D_1^-$ and $D_2^+$, $D_2^-$ respectively. One of the two possible smoothenings of the node $p$ connects $D_1^+$ with $D_2^+$ and the other one connects $D_1^+$ with $D_2^-$. Thus, if $C$ is dividing
then exactly one of the two deformations results in a dividing curve. Otherwise, both are non-dividing.
\end{proof}

Not every
vertical reflection in Figure~\ref{fig:eins} can be realized
geometrically by a discriminantal transition.
For instance,  the types $(91)1$ and $(11)9$ are related by a
vertical reflection. But both types are dividing, so
Proposition \ref{prop:turningout} implies that
they are not connected by turning inside out.
Put differently, these two components of $\PP^{27}_\RR \backslash \Delta$
 do not share a wall of codimension one. In fact, turning inside out can only happen for curves with at most $9$ ovals. Indeed, consider a plane sextic curve with exactly one crunode and $r$ connected components, one of which is the intersection of two pseudolines. The normalization of such a curve has genus $9$ and $r+1$ connected components. By Harnack's inequality this implies that $r\leq9$.

 We now come to the punchline of Section 4, namely, how the
 polynomials in Section~2 were created.  
 An established and powerful technique for constructing real varieties with prescribed
 topology is Viro's  {\em patchworking method}  \cite{Viro2}. 
 All $56$ topological types of smooth sextics can be
 realized by a version of patchworking known as {\em combinatorial
   patchworking}, which can also be interpreted in the language of
 tropical geometry. In that guise, one records the signs of the
 $28$ coefficients $c_{ijk}$ and represents their magnitudes by a
 regular triangulation of the Newton polygon.
 Transitioning from that representation to actual polynomials in
 $\ZZ[x,y,z]_6$ yields integer coefficients $c_{ijk}$ whose absolute
 values tend to be very large. We experimented with some of these
 sextics, but in the end we abandoned them for all but three types,
 because symbolic computation became prohibitively slow.

Instead, for deriving Proposition~\ref{prop:bigintegers},
we used that many types, especially the sextics with
$3$ to $7$ ovals, can be found by perturbing the union  of three quadrics intersecting transversally.
This process is shown
 in Figure \ref{fig:construction}. Some of the types with 8 and 9 components 
could also be constructed in a similar fashion.
This is reflected in our representatives in Section 2.

\begin{figure}[h]
\begin{center}
  \includegraphics[width=12cm]{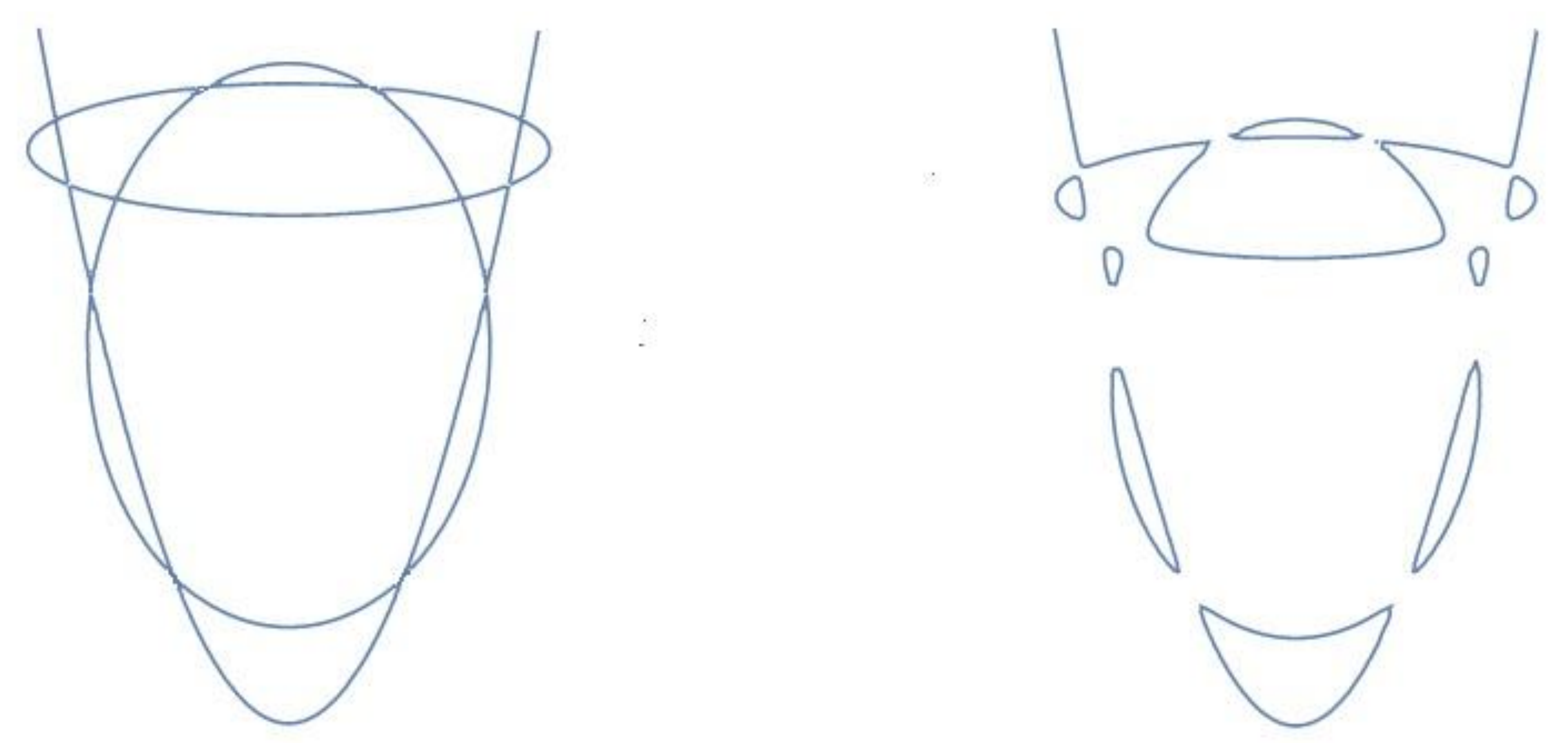}
  \end{center}
\vspace{-0.1in}
    \caption{\label{fig:construction} 
A sextic of Type $(11)7$ is constructed by perturbing the union of three quadrics.}
         \end{figure}

         For most other types, we carried out the classical
         constructions of Harnack and Hilbert, as explained by Gudkov
         \cite{Gud}.  We start with two quadrics intersecting in four
         real points,  pick eight points on the curves, and perturb
         the reducible quartic with the product of four lines through
         these points.  The smooth quartic is intersected with one of
         the original quadrics and perturbed again to get a smooth
         sextic.  The different ways in which the original curves and
         the points on them are selected give the different     types.          
         This method worked for almost all types.  For the
         construction of type (51)5, we tried Gudkov's method but found it
         too complicated to carry out explicitly. We employed Viro's method
         in \cite[\S 3.2]{Viro1} and
         built a curve of type (51)5 using patchworking on three touching ellipses.

\begin{figure}[h]
\begin{center}
\vspace{-0.2in}
  \includegraphics[height=8cm]{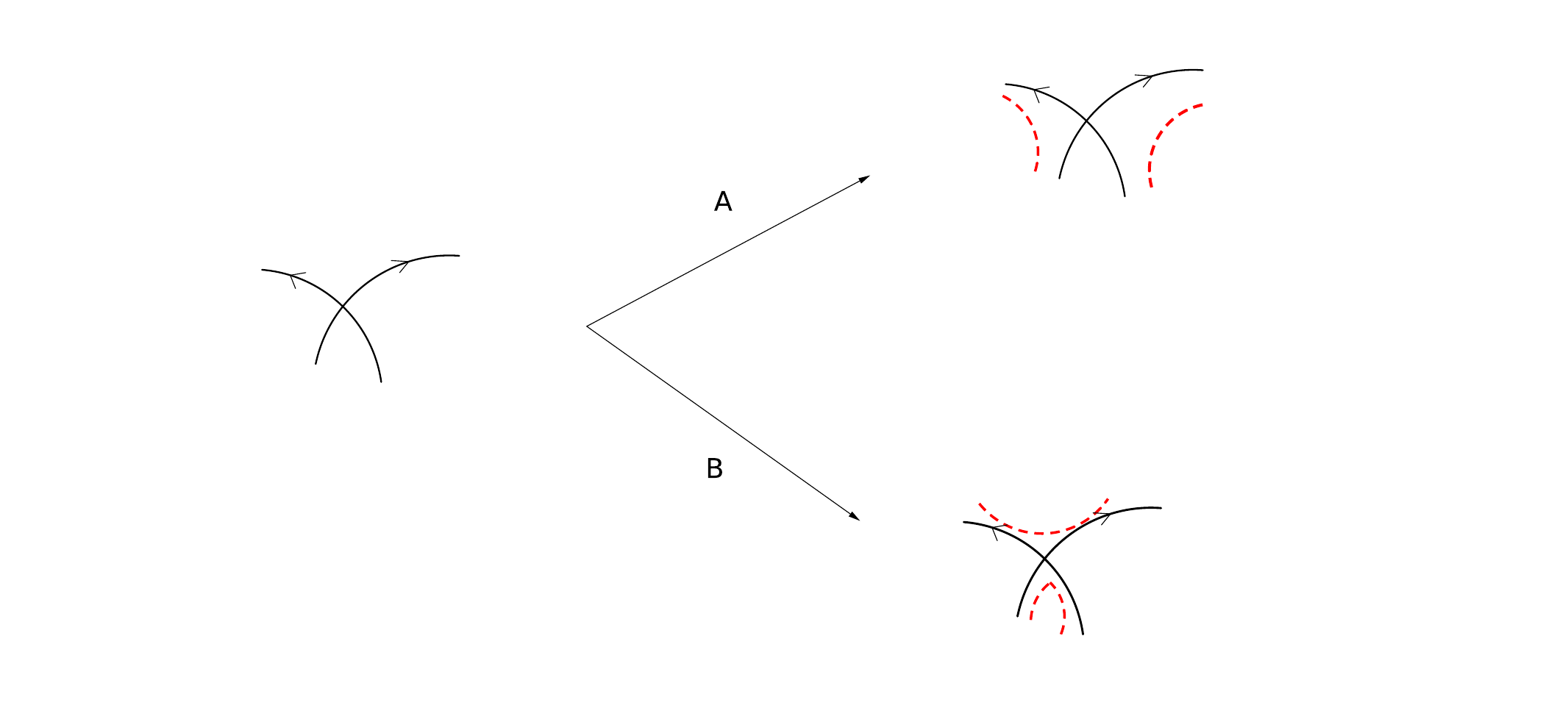}
  \end{center}
\vspace{-0.52in}
    \caption{\label{fig:drei} 
Using local perturbations to create sextics
  that are dividing or non-dividing
}
         \end{figure}
         
These constructions enabled us to find representatives
for the $56$ topological types. What remained was the issue
of distinguishing between dividing and non-dividing curves.
In particular, we needed to find two representatives for the
pairs of rigid isotopy types in the eight Nikulin cases (\ref{eq:8areboth}). To construct those, we
considered dividing curves of degree $d_{1}$ and $d_{2}$, with prescribed orientations,
that intersect in $d_{1}\cdotp d_{2}$ real points. The singular points of the reducible curve can be 
made smooth in two ways, shown in red in Figure \ref{fig:drei}.
According to Fiedler \cite[\S 2]{Fiedler}, if \emph{all} intersections 
are perturbed either using only  A or using only B then it is dividing. However, if the smoothening
is done  via A at some crossings
and  via B at other crossings, then the resulting smooth curve does not divide its 
 Riemann surface. In particular, for the construction of types (41)4d and (41)4nd we 
 followed \cite[page 273]{Gab3}.

At this point, we had $64$ representative sextics, and each 
of them was certified by our code {\tt SexticClassifier}. However,
the coefficient size for most of them was still unsatisfactory. To improve
the representatives, and to arrive at the list that is displayed in Section 2,
Sylvester's formula for the discriminant (Proposition \ref{prop:sylvester})
proved again to be very helpful.

 There are basically two heuristical methods that we used. One is to shrink the absolute value of each coefficient of a given representative separately as far as possible without crossing the discriminant locus. The other way is to choose a prime number $p$ and vary the polynomial without crossing the discriminant locus so that every coefficient of the resulting polynomial is divisible by $p$. In our experience, a combination of the two methods yields the best results.

\section{Avoidance Locus, Dual Curve, and Bitangents}

Many software packages for plane curves, such as those 
discussed at the beginning of Section~3, work with affine coordinates.
They often assume that the given curve is compact, so its closure
 in $\PP^2_\RR$ is disjoint from
a distinguished line, namely, the line at infinity.
We saw in  Example~\ref{ex:emptyavoidancelocus} that no such
line exists for some sextics. This motivates the
concept of the avoidance locus, to be introduced and studied in this section.
This will lead us naturally to computing dual curves
and bitangent lines, and to investigating the reality of these objects.

Let $C$ be a smooth real curve of even degree $d$ in $\PP^2$.
Its {\em avoidance locus} is the set
$\mathcal{A}_C$ of all lines in $\PP^2_\RR$ that do not intersect the real curve $C_\RR$.
Thus, $\mathcal{A}_C$ is a semi-algebraic subset of the dual
projective plane $(\PP^2)^\vee_\RR$. We write $C^\vee$ for the curve
of degree  $d(d-1)$ in $(\PP^2)^\vee$ that is dual to $C$. The points on $C^\vee$
correspond to lines in $\PP^2$ that are tangent to~$C$.
The real dual curve $C^\vee_\RR$ divides the real projective plane $(\PP^2)^\vee_\RR$
into connected components.

\begin{proposition}
Up to closure,
the avoidance locus $\mathcal{A}_C$ is a union of
connected components of $(\PP^2)^\vee_\RR \backslash C^\vee_\RR$.
Each component appearing in $\mathcal{A}_C$ is convex,
when regarded as a cone in $\RR^3$.
\end{proposition}

\begin{proof} 
Points in  $(\PP^2)^\vee_\RR \backslash C^\vee_\RR$
correspond to real lines that intersect $C$ transversally.
Whether that intersection contains real points or not 
does not change unless the curve $C^\vee_\RR$ is crossed.
Hence $\mathcal{A}_C  \backslash C^\vee_\RR$ is a union of 
connected components of  $(\PP^2)^\vee_\RR \backslash C^\vee_\RR$.
Each such component is convex: it is the convex dual of the
convex hull of $C_\RR$ in the affine space
$\PP^2_\RR \backslash L$, where $L \in \mathcal{A}_C$.
The prefix ``up to closure'' is needed because 
$\mathcal{A}_C$ also contains some points in $C^\vee_\RR$,
corresponding to real lines that do not meet $C_\RR$ but
are tangent to $C$ at complex points.
\end{proof}

  \begin{figure}[h]
  \begin{center}
     \includegraphics[width=5.6cm]{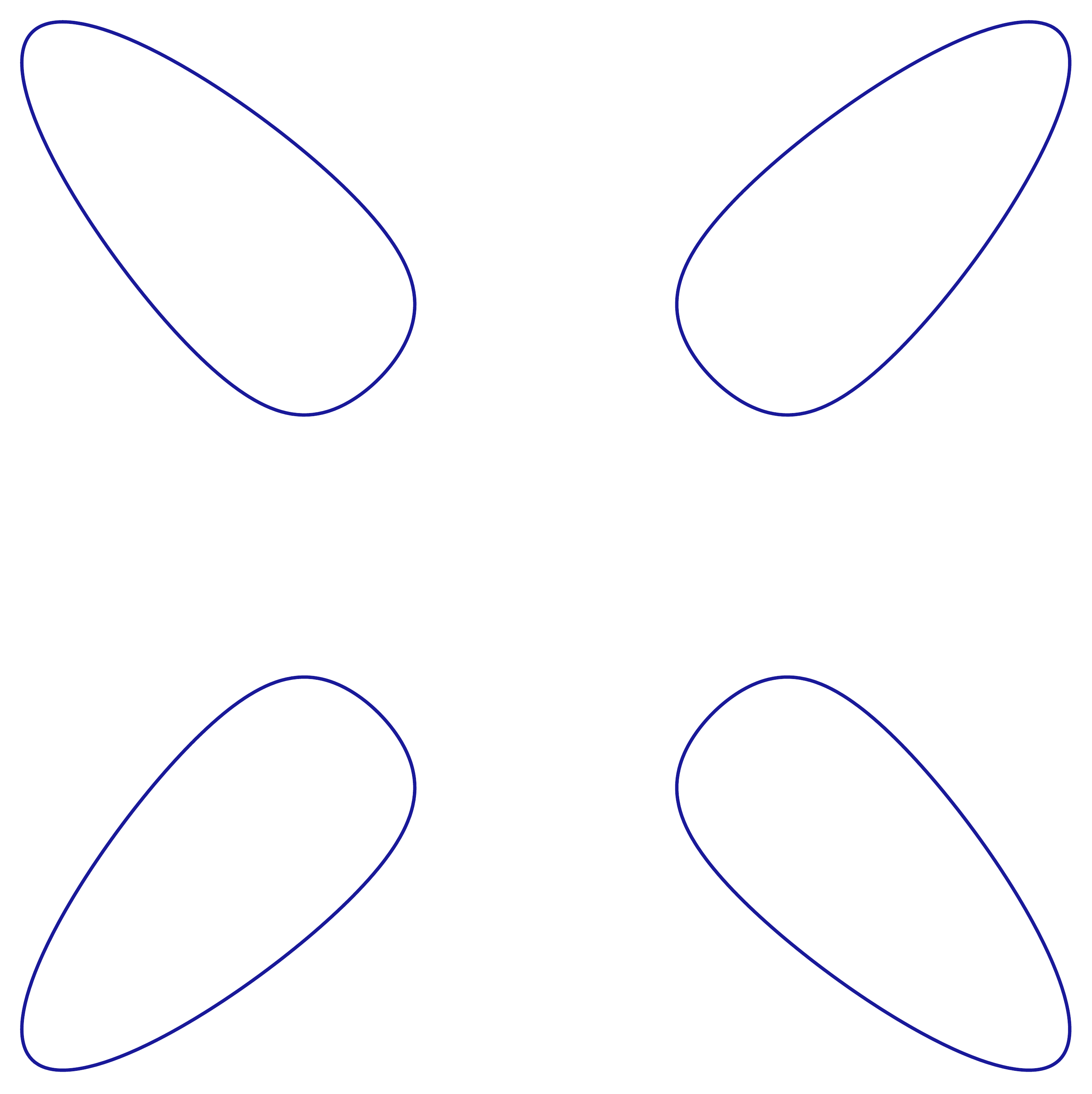} \qquad \qquad
    \includegraphics[width=5.7cm]{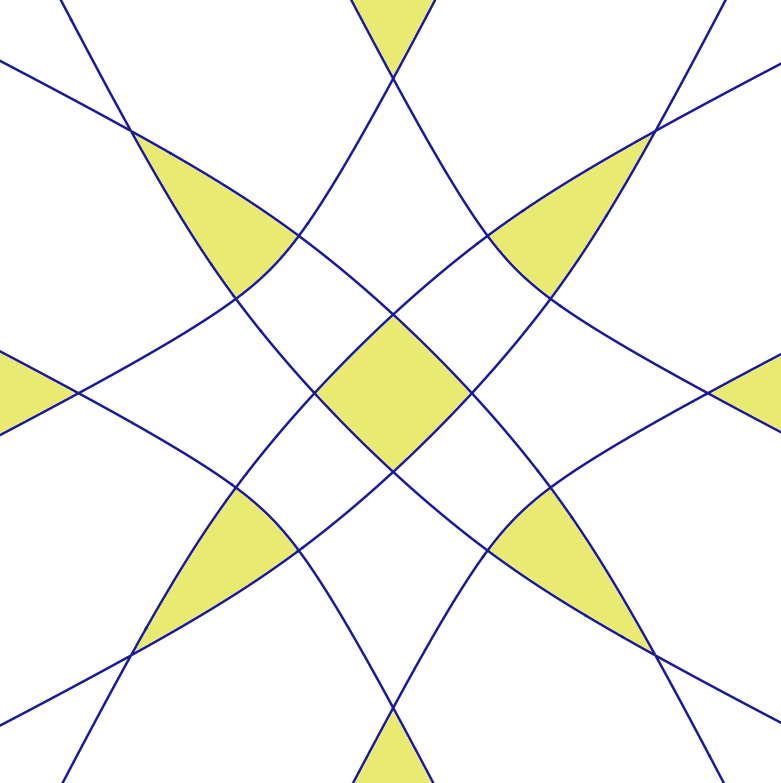}
     \vspace{-0.1in}
  \end{center}
   \caption{\label{fig:edge}   The Edge quartic $C$ and its dual $C^\vee$;
  the avoidance locus $\mathcal{A}_C$ is colored.}
  \end{figure}
  
\begin{example} \label{ex:p4} \rm
Let $d=4$ and consider the  \emph{Edge quartic} $C$,
taken from \cite[equation~(1.5)]{PSV}:
\[
  25 \left(x^4+y^4+z^4\right)-34 \left(x^2 y^2+x^2 z^2+y^2 z^2\right)\
  = \ 0.
\]
This curve $C \subset \PP^2$ served as a running example in \cite{PSV}.
It is shown on the left in Figure~\ref{fig:edge}.
We choose coordinates $(u:v:w)$ for points in the dual
projective plane $(\PP^2)^\vee$. Such a point represents the line
$L = \{ux+vy+wz=0\}$ in the primal $\PP^2$. 
The dual curve $C^\vee$ is given by

\vspace*{-.5em}
{\footnotesize
$$ \begin{matrix}
&10000 u^{12}-98600 u^{10} v^2-98600 u^{10} w^2+326225 u^8 v^4+85646
  u^8 v^2 w^2
+326225 u^8 w^4-442850 u^6 v^6\\
&-120462 u^6 v^4 w^2-120462 u^6 v^2 w^4-442850 u^6 w^6+326225 u^4
  v^8-120462 u^4 v^6 w^2+398634 u^4 v^4 w^4 \\ &
-120462 u^4 v^2 w^6 
  +326225 u^4 w^8-98600 u^2 v^{10}+85646 u^2 v^8
  w^2-120462 u^2 v^6 w^4-120462 u^2 v^4 w^6
\\ &+85646 u^2 v^2 w^8
-98600 u^2 w^{10}+10000 v^{12}-98600 v^{10} w^2+326225 v^8 w^4-442850
  v^6 w^6+326225 v^4 w^8 \\ &
-98600 v^2 w^{10}+10000 w^{12} \qquad = \qquad 0 . \qquad
\end{matrix} $$
}
The dual curve $C^\vee_\RR$ divides $(\PP^2)^\vee_\RR$
into $21$ open regions. Seven of the regions comprise the
avoidance locus $\mathcal{A}_C$. They are colored in
Figure \ref{fig:edge}, and they represent the
seven ways of bipartitioning the four ovals of $C_\RR$
by a straight line. The convex body dual to the
convex hull of $C_\RR$, in our affine drawing on the left,
is the innermost yellow region on the right.
\hfill $\diamondsuit$
\end{example}

The number seven of yellow regions seen in Figure~\ref{fig:edge}
attains the following upper bound.

\begin{proposition}
\label{prop:plane}
Let $C$ be a smooth real curve of even degree $d$ in $\PP^2$.
The number of open convex sets in the dual plane
that make up the avoidance locus $\mathcal{A}_C$ is bounded above by
\begin{equation}
\label{eq:p-polynomial}
\frac{9}{128}d^4-\frac{9}{32}d^3+\frac{15}{32}d^2-\frac{3}{8}d+1.
\end{equation}
\end{proposition}

\begin{proof}
  By Harnack's inequality, $C_\RR$ can have at most
  $\binom{d-1}{2}+1$ ovals in $\PP^2_\RR$.  However, for our count we
  only care about the outermost ovals, i.e.~those not contained inside
  any other oval. By a result of Arnold in \cite{Arn}, which is a more precise version of a classical inequality due to Petrovsky \cite{Pet}, the number of
  outermost ovals of the curve $C_\RR$ is at most
 \[
m\,\,=\,\,\frac 38 d^2-\frac 34 d+1.
\]
Pick a generic point in each oval.  Then the configuration of points has
$\binom{m}{2}+1$ bipartitions that can be realized by a straight line. 
Indeed, dually, this is the number of regions in the complement
of a general arrangement of $m$ lines in the plane $\PP^2_\RR$.  
The quartic polynomial in (\ref{eq:p-polynomial}) is simply
$\binom{m}{2}+1$ with Petrovsky's expression for $m$.
It remains to be seen that this number is the desired upper bound.
Indeed, every connected component  of $\mathcal{A}_C$
is uniquely labeled by a bipartition of the set of non-nested ovals.
The number of such bipartitions that are realized by a straight line
is bounded above by the said bipartitions of the points.
\end{proof}

The upper bound  in (\ref{eq:p-polynomial}) evaluates to $46$
for $d=6$.  Here is a sextic that attains the bound.

\begin{example} \label{ex:scheiderer} \rm
Let $t$ and $\epsilon $ be parameters, and consider  the following net of sextics:
\begin{align*}
F_{t,\epsilon}\ =\ &60x^6-750x^5z-111x^4y^2+1820x^4z^2+700x^3y^2z-2250x^3z^3+20x^2
y^4\\ &-1297x^2y^2z^2+960x^2z^4-56xy^4z+1440xy^2z^3-y^6-576y^2z^4\\
&+t(x^3+xz^2-y^2z)^2 + \epsilon(x^2z^4+y^2z^4+z^6)
\end{align*}
 For $\,t_0=-\frac{1645}{2}-150\sqrt{34}\,$
and $\,\epsilon=0\,$, the sextic $F_{t_0,0}$ has $10$ isolated real singular points:
\[
\bigl((3-\sqrt{34})/5:0:1\bigr), (0:0:1), (1:\pm\sqrt{2}:1), (2:\pm\sqrt{10}:1), (3:\pm\sqrt{30}:1),
(4:\pm 2\sqrt{17}:1).
\]
No three of these $10$ points lie on a line.
For any sufficiently small $\epsilon > 0$ and  $t$
sufficiently close to $t_0$, the sextic $F_{t,\epsilon}$ is
smooth with $10$ small ovals arranged around the singular
points of $F_{t_0,0}$. When these ovals are small enough, the
avoidance locus will have the maximum number $46$ of connected
components, by the argument given in the proof of
Proposition~\ref{prop:plane}.
This example was found using the construction developed by
Kunert and Scheiderer in \cite{KS}.
\end{example}

We now describe an algorithm for computing the avoidance locus 
$\mathcal{A}_C$ of a smooth 
curve $C$. The first step is to find all bitangents  of $C$. 
A {\em bitangent} of $C$ is a line $L$ in $\PP^2$ that
is tangent to $C$ at two points.
Note that bitangents of $C$ correspond to nodal singularities of the dual curve $C^\vee$.
By the Pl\"ucker formulas, the
expected number of bitangents is $ (d-3)(d-2)d(d+3)/2$, which is $324$ for $d=6$.
A bitangent $L$ is called \emph{relevant} if the real part of the
divisor $L\cap C$ is an even divisor on the curve $C$. For generic curves $C$, 
this means that $L$ has no real intersection points with $C$ except
possibly the two points of tangency. If these two points are real, then
$L$ is an extreme point of a convex connected component of $\mathcal{A}_C$.

\begin{remark}\label{rem:relevantbitangents} \rm
Let $C$ be a smooth curve in $\PP^2_\RR$ of  degree $d \geq 4$. If $C$
contains at least two outermost ovals or has a non-convex outermost oval, then
every connected component of the avoidance locus 
$\mathcal{A}_C$ has a relevant bitangent in its closure.
If $C$ does not satisfy this hypothesis then
$\mathcal{A}_C$ is connected;
  we do not know  whether
  it always contains a real bitangent. In the case of quartics, this follows from the Zeuthen classification \cite[Table 1]{PSV}.
   In that case, however, the number of bitangents only depends on the topological type. In higher degrees, when this is not the case, not much seems known. (See Conjecture \ref{conjecture:realbitangents} below.)
\end{remark}

We now assume that $C$ satisfies the hypothesis in
Remark~\ref{rem:relevantbitangents}.  Our algorithm for computing
$\mathcal{A}_C$ is as follows.  First we compute linear forms
representing all real bitangents of $C$, and we discard those that are not
relevant.  Next, we compute a graph $\mathcal{G}_C$ whose nodes are the
relevant bitangents,
as follows: two linear forms $L_1$ and $L_2$ form an
edge if and only~if
\begin{enumerate}
\item[(a)] $L_1 + L_2$ lies in the avoidance locus $\mathcal{A}_C$, and \vspace{-0.1in}
\item[(b)] the open line segment $\{t L_1 +  (1-t) L_2 \,:\, 0 < t < 1 \}$ is disjoint
from the dual curve~$C_\RR^\vee$.
\end{enumerate}
Here the sign of the linear forms $L_1$ and $L_2$ for
the bitangents has to be chosen carefully.

\begin{remark}{\rm 
The graph $\mathcal{G}_C$ is a disjoint union of cliques,
one for each connected component of $\mathcal{A}_C$.
This follows from Remark~\ref{rem:relevantbitangents}
and convexity of the connected components.}
\end{remark}

In summary, given a smooth curve $C=V_\CC(f)$ of even degree $d$,
our algorithm computes the {\em avoidance graph} $\mathcal{G}_C$.
We represent the avoidance locus $\mathcal{A}_C$ by the
connected components (cliques) of $\mathcal{G}_C$. 
Midpoints of the segments in~(b) furnish
sample points in the components.

We made a proof-of-concept implementation of this algorithm
for the case of  sextics. Its two main ingredients are computing
the dual curve and computing the bitangents.
For the former task we solve a linear system of equations
in the $\binom{30+2}{2} = 496$ coefficients of $C^\vee$.
The equations are derived  by projecting $C $
from random points $p \in \PP^2$. The ramification locus
of this projection reveals (up to scaling) the binary form of degree $30$ that defines
 $ C^\vee \cap p^\perp $.
 
  To compute the bitangents, we employ the
variety of binary sextics with two double roots.
The prime ideal of this variety is defined by $13$ forms of degree $7$;
see the row labeled 2211 in \cite[Table 1]{LeSt}.
Substituting the binary form $f(x,y,  - \frac{1}{w}(u x + v y)) $ into
that ideal, and clearing denominators, yields the ideal in
$\QQ[u,v,w]$ that defines the $324$ bitangents $(u{:}v{:}w) \in (\PP^2)^\vee $.

\begin{example} \rm 
Let $C$ be the representative for Type 8nd displayed in Section 2.
This sextic curve has $324$ distinct complex bitangents
 of which $124$ are real. Of the real bitangents,
  \begin{itemize}
    \item $8$ are tangent at non-real points and meet the curve
    in two more non-real points; \vspace{-0.1in}
    \item $60$ are tangent at real
    points and meet the curve in two more non-real points; \vspace{-0.1in}
    \item $4$ are tangent at non-real points and meet the curve in two
    more real points; \vspace{-0.1in}
    \item $52$ are tangent at real points and
    meet the curve in two more real points.
  \end{itemize}
  Only the first two types are relevant, so $C$ has $68$ relevant bitangents.
The avoidance graph $\mathcal{G}_C$ is found to
consist of $14$ cliques: four $K_6$'s, five $K_5$'s, four $K_4$'s and one $K_3$.
Hence $\mathcal{A}_C$ consists of $14$ convex components.
  The curve $C$ together with its
   $68$ relevant bitangents is shown in Figure \ref{fig:avoidance8nd}.
   There are $14$ ways to bipartition the $8$ ovals by a line that avoids $C_\RR$.
  \begin{figure}[h]
  \begin{center}
     \includegraphics[width=10cm]{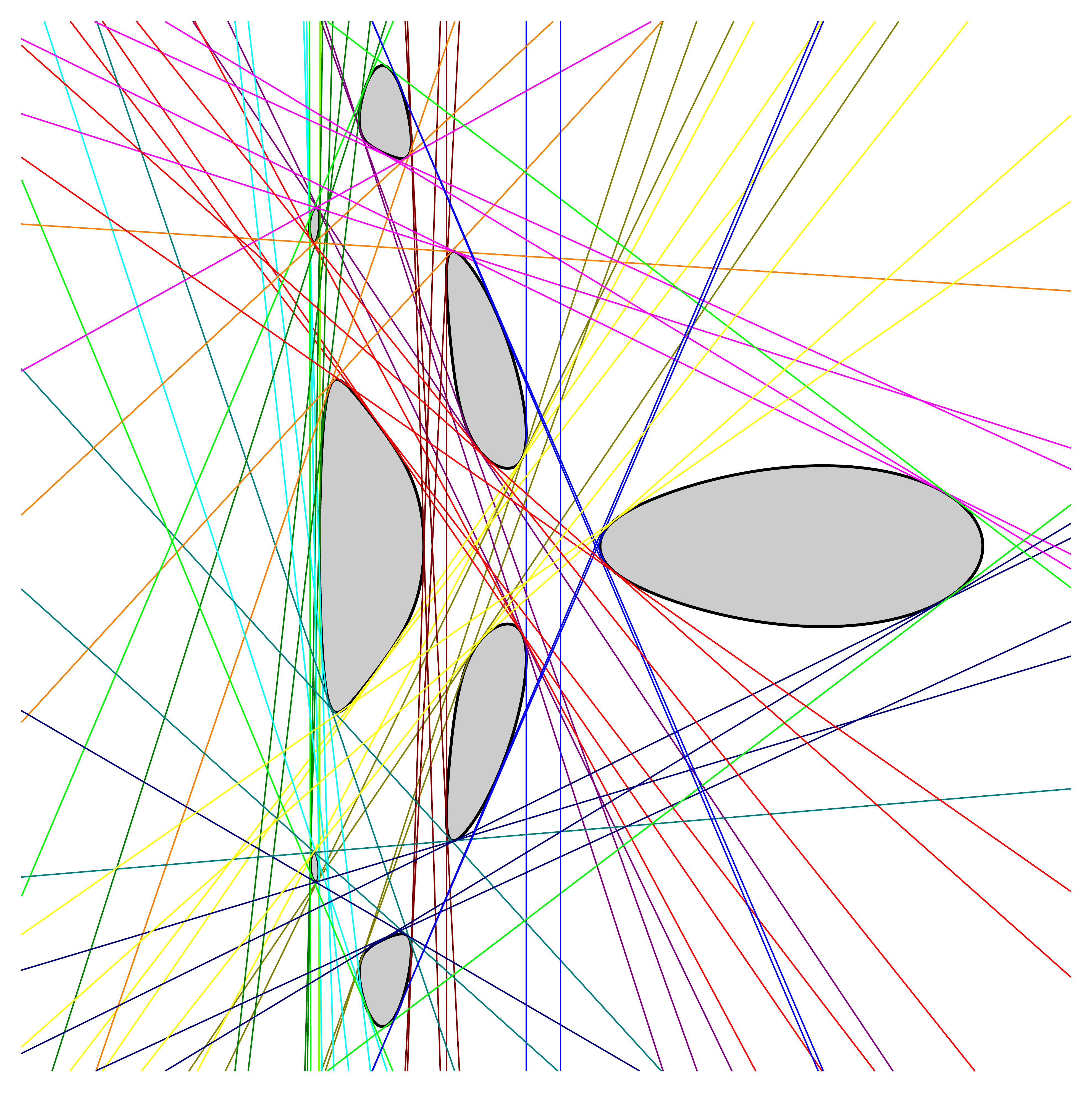}
  \end{center}
  \vspace{-0.2in}
  \caption{\label{fig:avoidance8nd} A sextic curve $C$   with $8$
    non-nested ovals; its $68$ relevant bitangents represent~$\mathcal{A}_C$.}
          \end{figure}
\end{example}

While the nodes on $C^\vee$ are bitangents of $C$,
the cusps on $C^\vee$ are the {\em flex lines} of $C$.
The number of {\em inflection points} is $3d(d-2)$ for a general curve
of degree $d$.
A classical result due to Felix Klein states that at most one third
of the complex inflection points of a real plane curve 
can be real.  Brugall\'e and L\'opez de Medrano \cite{BL} proved,
using tropical methods, that Klein's upper bound $d(d-2)$ is attained for all $d \geq 3$.
Hence, for smooth sextics, the number of real inflection points can
be any even integer between $0$ and $24$. 
The distribution of the numbers of real bitangents and real inflection points
over the $64$ rigid isotopy types is presented in Section~6.
 Based on our experiments, we propose the following conjecture:

\begin{conjecture}\label{conjecture:realbitangents}
  The number of real bitangents of a smooth sextic in $\PP^2_\RR$
  ranges from 12 to 306. The lower bound is attained by curves in the
  following four types: empty, {\rm 1}, {\rm 2}, {\rm (11)} and {\rm (hyp)}.  The upper bound is
  attained by certain $11$-oval curves of Gudkov-type (51)5.
\end{conjecture}


The numbers of real inflection points and bitangents of a sextic $C$
often change when passing through the discriminant
hypersurface. However, they may also change within the same rigid
isotopy type. If $C$ is smooth then the total number of complex
inflection points resp.~bitangents drops below the bounds $72$
resp.~$324$ in the following three exceptional cases:
  \begin{enumerate}
  \item[(411)] $C$ has an \emph{undulation point}, in which the tangent
    meets $C$ with multiplicity at least $4$. \vspace{-0.1in}
  \item[(222)] $C$ has a \emph{tritangent line}, i.e.~a line that is tangent
    to $C$ in three distinct points. \vspace{-0.1in}
  \item[(321)] $C$ has a \emph{flex-bitangent}, i.e.~a line that meets $C$
    with multiplicity $3$ in one point and is tangent at another point.
  \end{enumerate}
  In each case, the stated property defines a
  hypersurface in $V = \RR[x,y,z]_6$. The number of real inflection
  points or bitangents changes only when passing through the
  discriminant or one of these hypersurfaces.  When generic sextics
  approach these hypersurfaces, three lines come together: a
  bitangent and two flex lines for an undulation point (411), three
  bitangents for a tritangent line (222), and two bitangents and a
  flex line for a flex-bitangent (321).

\begin{theorem}\label{thm:bitangentdisc}
  Let $\mathcal{T}$ be the Zariski closure in $\PP V = \PP^{27}$ of the set of smooth
  sextics  with a tritangent line and
    let $\mathcal{F}$ be the locus of smooth sextics with a
  flex-bitangent.  Then:
  \begin{enumerate}
    \item The loci $\mathcal{T}$ and $\mathcal{F}$ are irreducible hypersurfaces of
      degree $1224$ and $306$ respectively.
    \item Their union $\mathcal{B}=\mathcal{T}\cup\mathcal{F}$ is the
      {\em bitangent discriminant}, i.e.~the Zariski closure of the set of smooth
  sextics in $\PP V$ having fewer than $324$ bitangent lines.
  \end{enumerate}
\end{theorem}

\begin{proof}
The variety of binary sextics with three
  double roots is irreducible of codimension $3$ in $\RR[x,y]_6$.
  (It is defined by $45$ quartics \cite[Table 1]{LeSt}.)
   Let $\mathcal{X}$ be the incidence
  variety of all pairs $(L,f)$ in $(\PP^2)^\vee\times\PP V$ 
  where $L$ is a tritangent of $V_\CC(f)$.
  The locus $\mathcal{T}$ is the projection of $\mathcal{X}$ onto $\PP V$.
   The  intersection of $\mathcal{X}$ with any subspace of the form
  $\{L\}\times\PP V$ for $L\in(\PP^2)^\vee$ has codimension  $3$ in
  $\PP V$. 
     Taking the union over all $L$, we
  conclude that $\mathcal{T}$ has codimension $1$. Since the
  projection of $\mathcal{X}$ onto the first factor is surjective with
  irreducible fibers of constant dimension, $\mathcal{X}$ is
  irreducible, hence so is $\mathcal{T}$. The same argument applies to~$\mathcal{F}$.
The degrees of the hypersurfaces $\mathcal{F}$ and $\mathcal{T}$
were computed for us by
Israel Vainsencher with the {\tt Maple} package {\tt schubert}.
The relevant theory is described by Colley and Kennedy in~\cite{CK}.

  To prove (2), we first note that a tritangent splits into three
  bitangents, and a flex-bitangent into two bitangents and a flex line,
  for any smooth deformation of a sextic in $\mathcal{T}$ or
  $\mathcal{F}$, respectively. This shows that $\mathcal{T}$ and
  $\mathcal{F}$ are both contained in the bitangent discriminant. For
  the reverse, we argue in the dual picture, with degenerations of
singularities on  $C^\vee$. \end{proof}

  \begin{figure}[h]
  \begin{center}
  \vspace{-0.1in}
     \includegraphics[width=9.5cm]{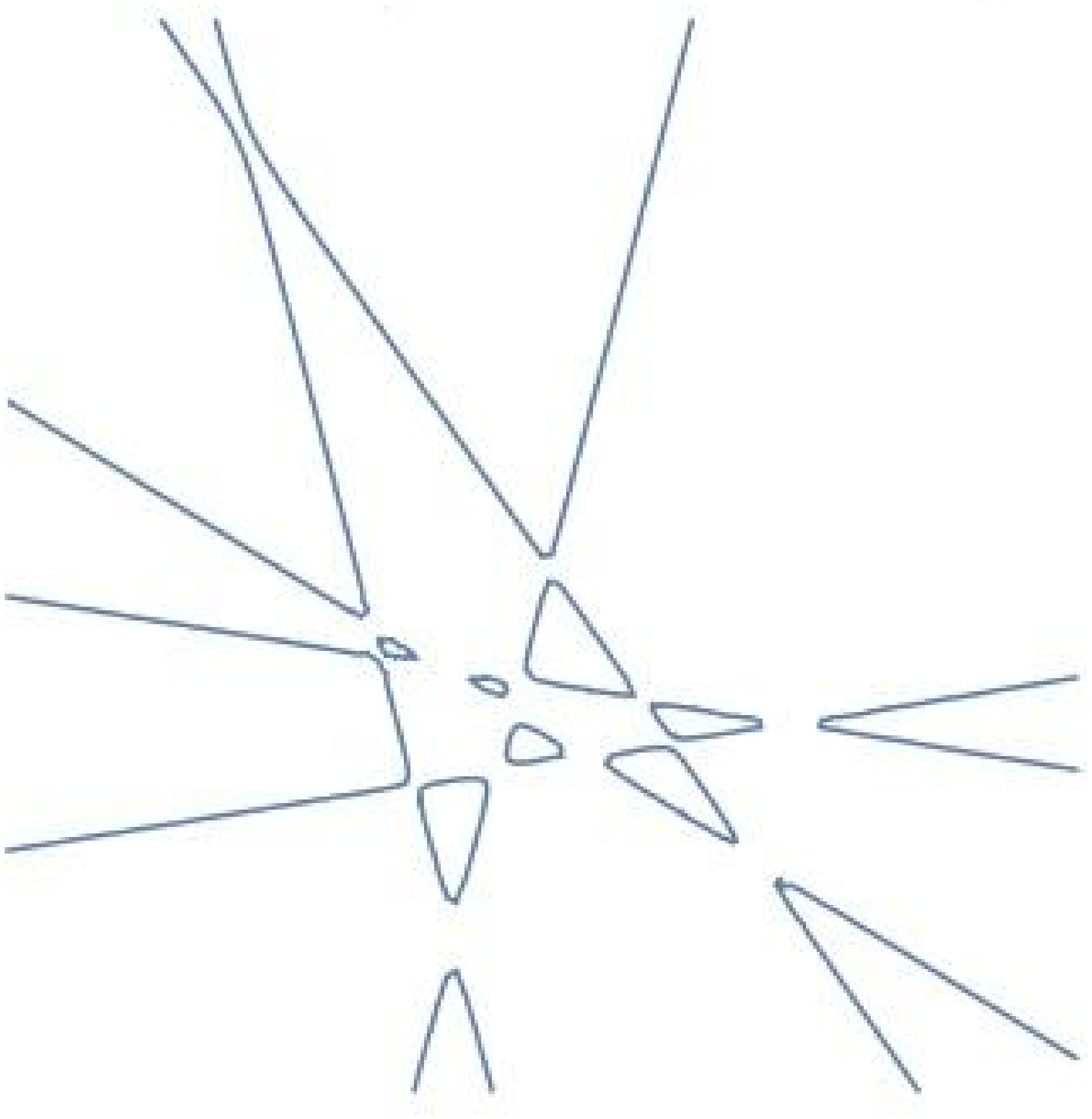}
          \vspace{-0.3in}
  \end{center}
   \caption{\label{fig:mahsa10}  
A smooth sextic with $10$ non-nested ovals whose avoidance locus is empty}
  \end{figure}

We conclude this section with the following result on the avoidance loci
of plane sextics.

\begin{corollary}
  For any integer $m$ between $0$ and $46$, there exists a smooth
  sextic $C$ in $\PP^2_\RR$ whose avoidance locus $\mathcal{A}_C$ 
  comprises exactly $m$ convex connected components.
\end{corollary}
  
\begin{proof}
  Let $f_1$ be a sextic of Type 10nd whose avoidance locus has $46$
  components, as in Example \ref{ex:scheiderer}. Let $f_0$ be
  a sextic of Type 10nd with empty avoidance locus, for instance
 $$ f_0 \,\,=\,\, \bigl(20x^2-(y{+}10z)^2+z^2\bigr)\bigl(21y^2-(x{-}10z)^2+z^2\bigr)\bigl( 20 (x-5z+y)^2
 -(y{+}10z)^2+z^2\bigr)+z^6 .$$
 The real picture of such a curve is shown in Figure \ref{fig:mahsa10}.
       Let $U$ be the connected component of
  $\PP^{27}_\RR \backslash \Delta$ that contains both $f_0$ and $f_1$.
  Consider the bitangent discriminant $\mathcal{B}$ of Theorem
  \ref{thm:bitangentdisc}. There is an open dense subset
  $\mathcal{B}_0$ of $\mathcal{B}$, whose points represent curves with
  a single tritangent line or a single flex-bitangent. The
  exceptional locus $Z=\mathcal{B}\setminus\mathcal{B}_0$ has
  codimension at least $2$ in $\PP^{27}_\RR$, hence $V=U \backslash Z$
  is path-connected. Fix a path
  $\gamma\colon [0,1]\to V$ with $\gamma(0)=f_0,
  \,\gamma(1)=f_1$. Let $\{\gamma(t_1),\dots,\gamma(t_k)\}$ be its
  intersection points with $\mathcal{B}$. Since
  $\gamma(t_i)$ lies in $\mathcal{B}_0$, the number of
  connected components of the avoidance locus of $V_\RR(\gamma(t))$
  cannot change by more than $1$ in a neighborhood of $t_i$. Indeed,
  at a point where that number drops by
  one, exactly three relevant bitangents come together, giving a
  sextic with a single tritangent in $\mathcal{B}_0$.
    Hence any number of convex avoidance components
  between $0$ and $46$ is realized along the path $\gamma$.
\end{proof}

\section{Experiments and K3 Surfaces}

We begin this section with a report on further experiments on
reality questions for plane sextics. Our findings are summarized
in Table \ref{tab:flexetc}. Thereafter, we discuss implications
for real K3 surfaces, and we show how to construct
 quartic surfaces in $\PP^3_\RR$ with prescribed topology.

Table \ref{tab:flexetc} summarizes experimental data on the numbers
of real features associated with real sextics in $\PP^2_\RR$.
Certified samples were drawn in the vicinity
 of each of our current $64$ representatives,
using the local exploration method that is
described after Proposition \ref{prop:sylvester}.

Each row of Table \ref{tab:flexetc}
has five entries: the name of the 
rigid isotopy type, numbers of real inflection points,
numbers of real eigenvectors, numbers of real bitangents,
and one real rank.  The numbers are ranges
of integers that were observed in our experiments. For instance,
for Type $(11)1$,  we found numbers ranging between $20$ and $66$  of
real bitangents among $324$ complex ones.  In some cases,
all samples gave the same number of real solutions.
For instance, all our Type $(71)$ sextics had  $108$ 
real bitangents. 
Each entry in Table~\ref{tab:flexetc} can be regarded
as a {\bf conjecture}. For example, we conjecture that every
smooth sextic of Type $(71)$ has exactly $108$ real bitangents.
For real inflection points and real eigenvectors we performed exact
computations. That means that we are sure that all numbers listed in
the table actually occur. However, we do not know whether there are more possible numbers.
For bitangents the calculations are more involved and we applied numerical methods. This means that, for some instances, the number of real solutions might be undercounted.
This can happen when two or more real bitangents lie very close to each other. The computations for estimating the real rank are even more delicate and were also accomplished numerically.

Bitangents and their applications
were discussed in detail in the previous section.
We now give the definitions needed to understand
the other three columns in Table~\ref{tab:flexetc}.

The second column concerns {\em inflection points} of
a smooth sextic $V_\CC(f) \subset \PP^2$. There are
 $72 = 6 \cdot 12 $ complex inflection points. They are computed as the
 solutions of the equations
$$ f(x,y,z) \,\,= \,\, {\rm det} \begin{pmatrix}
 \frac{\partial^2 f}{\partial x^2} &
 \frac{\partial^2 f}{\partial x \partial y} &
 \frac{\partial^2 f}{\partial x \partial z} \smallskip \\
\frac{\partial^2 f}{\partial x \partial y} &
 \frac{\partial^2 f}{\partial y^2} &
 \frac{\partial^2 f}{\partial y \partial z} \smallskip \\
\frac{\partial^2 f}{\partial x \partial z} &
 \frac{\partial^2 f}{\partial y \partial z} &
 \frac{\partial^2 f}{\partial z^2} 
\end{pmatrix} 
\,\, = \,\,\, 0 .
$$
A curve of degree $d$ has $3d(d-2)$ complex inflection points.  A
classical result due to Felix Klein states that the number of real
inflection points is at most $d(d-2)$. Brugall\'e and L\'opez de
Medrano \cite{BL} showed that this upper bound is tight for all
degrees $d$. Hence, for a general sextic in $\PP^2_\RR$, the number of
real inflection points is an even integer between $0$ and $24$. The
column labeled ``Flex'' shows the empirical distribution on the $64$
rigid isotopy types.

\begin{table}
$$ \begin{small}
\begin{matrix} 
{\rm Type} 	&               {\rm Flex}		&    {\rm Eigenvec}     & 	{\rm Bitang}		&	{\rm Rank}	\\
     0     	&  		0			&	3{-}31  	&	12			&	3		\\
     1     	& 		0{-}12			&	3{-}31^*	&	12{-}56			&	3		\\
     (11)  	& 		0{-}14			&	11{-}31^* 	&	12{-}66			&	10		\\
     2     	&		0{-}8			&	5{-}31^*	&	12{-}52			&	13		\\
     (21) 	&		0{-}10			&	7{-}31^*	&	16{-}86			&	14		\\
     (11)1 	&		2{-}6			&	7{-}31^*	&	20{-}66			&	15		\\
      3    	&		0{-}8			&	7{-}31^*	&	24{-}94			&	13		\\
     ({\rm hyp})& 	 	0{-}14			&	11{-}31^*	&	12{-}52			&	13		\\
     (31)  	&		2{-}10			&	19{-}31^*	&	24{-}90			&	13		\\
     (21)1 	&  		0{-}6			&	11{-}31^*	&	28{-}72			&	14		\\
     (11)2 	& 		0{-}4			&	11{-}31^* 	&	32{-}82			&	13		\\
       4   	&  		0{-}2			&	11{-}31^*	&	36{-}54			&	11		\\
(41){\rm nd}	&		14{-}16			&	21{-}31^*	&	48{-}90			&	16		\\
(41) {\rm d}	& 		12{-}14			&	27{-}31^*	&	98{-}104		&	14		\\
     (31)1  	& 		2{-}8			&	15{-}31^*	&	40{-}86			&	14		\\
(21\!)2{\rm nd}	&		10{-}16			&	17{-}31^*	&	54{-}82			&	20		\\
(21)2{\rm d} 	& 		8{-}16			&	19{-}31^*	&	60{-}70			&	17		\\
    (11)3 	& 		8{-}12			&	19{-}31^*	&	48{-}94			&	14		\\
       5  	& 		2{-}10			&	19{-}31^*	&	52{-}112			&	15		\\
    (51) 	&		12{-}16			&	21{-}31^*	&	54{-}64			&	14		\\
   (41)1 	&		22			&	27{-}31^*	&	90{-}104		&	14		\\
   (31)2 	&		14{-}18			&	27{-}31^*	&	126{-}130		&	14		\\
   (21)3 	&		16			&	27{-}31^*	&	112{-}116		&	14		\\
   (11)4 	& 		6{-}10			&	25{-}31^*	&	76{-}106		&	15		\\
     6  	& 		10{-}12			&	23{-}31^*	&	78{-}108			&	14		\\
   (61) 	&  		16			&	27{-}31^*	&	78{-}88			&	14		\\
(51\!)1{\rm nd} &	 	16			&	23{-}25		&	110{-}124			&	15		\\
(51)1 {\rm d}	&  	 	20{-}24			&	29		&	136			&	16		\\
     (41)2   	&		16{-}20			&	29{-}31		&	126{-}128		&	14		\\
(31\!)3{\rm nd} &		12			&	25{-}31^* 	&	124{-}148		&	15		\\
(31)3{\rm d}	&		20{-}22			&	29		&	132			&	16		\\
    (21)4 	& 		14{-}20			&	27{-}31^*	&	138{-}142		&	15		\\
         \end{matrix} \qquad \quad
     \begin{matrix}
 {\rm Type} 	 &  		{\rm Flex}		&    {\rm Eigenvec}     & 	{\rm Bitang}		&	{\rm Rank} \\
(\!11\!)5{\rm nd}& 		6{-}16			&	29{-}31^* 	&	116{-}122		&	16			\\
   (11)5  {\rm d}&		8{-}16			&	25{-}31^* 	&	120{-}128		&	16			\\
     7 		 &        	4{-}14			&	25{-}31^*	&	96{-}124		&	14			\\
     (71) 	 &  		20{-}24			&	29		&	108			&	16			\\
     (61)1  	 &  		20{-}22			&	25		&	104{-}214		&	15			 \\
     (51)2  	 &  		22			&	25{-}31		&	226{-}228			&	15			\\
     (41)3   	 & 		20			&	23{-}25		&	154{-}214		&	14			 \\
     (31)4  	 & 		22			&	21		&	162{-}214			&	14			 \\
      (21)5 	 & 		16{-}20			&	29{-}31	 	&	168			&	13			 \\
     (11)6  	 & 		12{-}14			&	27{-}31^*	&	172{-}176			&	14			 \\
         8  	 &  		0{-}12			&	23{-}31^*	&	124{-}142			&	13			\\
(81) {\rm nd} 	 & 		18{-}22			&	23		&	122{-}196			&	14			\\
(81)  {\rm d} 	 & 		18{-}24			&	29		&	124{-}132			&	12			 \\
     (71)1  	 &  		14{-}18			&	21{-}31		&	104{-}240		&	13			\\
     (61)2  	 &  		18{-}20			&	23{-}31		&	228{-}276			&	13			\\
     (51)3  	 &  		22			&	25	 	&	192{-}254			&	13			 \\
(41)4 {\rm nd} 	 &		14{-}16			&	25		&	188{-}220			&	 9			\\
(41)4  {\rm d} 	 & 		18			&	25		&	194{-}230			&	11			 \\
     (31)5  	 &  		20			&	25{-}31		&	198{-}260			&	13			 \\
     (21)6  	 &  		20			&	23{-}31		&	242{-}258		&	15			 \\
     (11)7  	 &		14{-}16			&	29{-}31		&	216			&	14			 \\
 9 {\rm nd} 	 & 		8{-}16			&	25{-}31^*	&	162{-}172		&	15			 \\
 9  {\rm d} 	 & 		4{-}16			&	29{-}31^*	&	156			&	15				\\
       (91) 	 & 		18{-}22			&	23		&	124{-}236			&	13			 \\
      (81)1 	 & 		16{-}20			&	23{-}31		&	162{-}240			&	14			 \\
      (51)4 	 & 		20			&	27	 	&	232{-}234			&	10			 \\
      (41)5 	 &  		18{-}20			&	27{-}31   		&	232			&	10			\\
     (11)8  	 &  		14{-}18			&	25{-}31		&	142{-}210			&	13			\\
        10  	 &  		0{-}24			&	21{-}31^*	&	192			&	12			\\
     (91)1  	 &  		18{-}22			&	25{-}31		&	200{-}284			&	14			\\
     (51)5  	 &  		20{-}22			&	25{-}31		&	276{-}306		&	10			\\
     (11)9  	 &   		16{-}20			&	25{-}31		&	174{-}250		&	14			\\
\end{matrix}
\end{small}
$$
\caption{Computational results for the number of real solutions
for inflection points, eigenvectors, bitangents and real rank
among the $64$ rigid isotopy classes of smooth sextics in $\PP_\RR^2$.
\label{tab:flexetc}}
\end{table}

The third and fifth column pertain to the study of tensors in multilinear algebra.
Here we identify the space $ \RR[x,y,z]_6$ of ternary sextics
with the space of symmetric tensors of format
$3 {\times} 3 {\times} 3 {\times} 3 {\times}3 {\times} 3 $.
Such a symmetric tensor $f$ has $28$ distinct entries, and these
are the coefficients $c_{ijk}$ of the sextic. A vector $v \in \CC^3$
is an {\em eigenvector} of $f$ if $v$ is parallel to the gradient
of $f$ at $v$. Thus the eigenvectors correspond to the solutions
in $\PP^2_\CC$ of the constraint
\begin{equation}
{\rm rank} \begin{pmatrix} x & y & z \smallskip \\
\frac{\partial f}{\partial x} & \frac{\partial f}{\partial y} & \frac{\partial f}{\partial z} 
\end{pmatrix} \,\,\, = \,\,\, 1 .
\label{eq:eigvec} 
\end{equation}

A general ternary form $f$ of degree $d$ has $ d^2 -d + 1$ eigenvectors \cite[Theorem 2.1]{ASS}.
The eigenvectors are the critical points of the
optimization problem of maximizing $f$
on the unit sphere $\mathbb{S}^2 = \{(x,y,z) \in \RR^3: x^2+y^2 + z^2 = 1\}$.
Since $f$ attains a minimum and a maximum on $\mathbb{S}^2$, the
number of real eigenvectors is at least $2$. 
We note that the upper bound $d^2 - d+1$ is attained over $\RR$.
If $f$ is a product of $d$ general linear forms, then
all its complex eigenvectors are real. This was shown in \cite[Theorem 6.1]{ASS}.
For $d=6$, we conclude 
 that the number of real eigenvectors of a general ternary
sextic is an odd integer between $3$ and~$31$.
The column labeled ``Eigenvec'' shows the empirical distribution
on the rigid isotopy types. For many rigid isotopy types we found instances that attain the maximal number $31$ of real eigenvectors. 
Among them are the $35$ types that have unions
 of six real lines in general position in their closure. These types are
marked with an asterisk next to the number $31$.
 We found these by   perturbing each 
 of the four combinatorial types of arrangements of six lines in general 
 position in~$\mathbb{P}_\mathbb{R}^2$.
 This search process resulted in   $35$ of the rigid isotopy types.
 This is the result stated in Proposition \ref{prop:thirtyfive}.
 The computation we described is the proof.

A theoretical study of real eigenvectors was undertaken by
Maccioni in \cite{Mac}. He proved that the number of real eigenvectors 
of a ternary form is bounded below by $2 \omega+ 1$, where 
$\omega$ is the number of ovals.
Our findings in the third column of Table~\ref{tab:flexetc}
confirm  this theorem. Moreover, there are seven types where our computations prove the converse, namely that all values between this lower bound and the upper bound $31$ are realized in these types.

Every tensor is a sum of rank one tensors. The smallest number of
summands needed in such a representation is the {\em rank} of
that tensor. This notion depends on the underlying field.
Symmetric tensors of rank $1$ are powers of linear forms (times a
constant). Hence, the rank $r$ of a ternary form $f$ of degree $d$ is the 
minimum number of summands in a representation
\begin{equation}
\label{eq:frank}
 f(x,y,z) \quad = \quad
\sum_{i=1}^r \lambda_i (a_i x + b_i y + c_i z)^d.
\end{equation}
The exact determination
of the real rank of a sextic $f$ is very difficult. The task is
 to decide the solvability over $\RR$
of the equations
in the unknowns $\lambda_i,a_i,b_i,c_i$
obtained by equating coefficients in (\ref{eq:frank}).
This computation is a challenge for both symbolic and numerical methods.
There is no known method that is guaranteed to succeed in practice.
If $f$ is a generic sextic in $\RR[x,y,z]_6$ then the complex rank of $f$ is $10$,
and the real rank of $f$ is an integer between $10$ and $19$.
This was shown in
\cite[Proposition 6.3]{MMSV}.
This upper bound is probably not tight. 

We experimented with the software {\tt tensorlab} \cite{tensorlab}.
This is a standard package for tensors, used 
in the engineering community. This program furnishes
a local optimization method for the following problem:
given $f$ and $r$, find a sextic $f^*$ of rank $r$ that is closest to $f$,
with respect to the Euclidean distance on the tensor space $(\RR^3)^{\otimes 6}$.
If the output $f^*$ is very close to the input $f$, we can be confident that $f$
has real rank $\leq r$. If $f^*$ is far from $f$, even after many tries
with different starting parameters,  then we believe
that $f$ has real rank $\geq r+1$. However, {\tt tensorlab} does not
furnish any guarantees. One needs to rerun the same instance
many times to achieve a lower bound on the real rank with high confidence.

The last column  of Table \ref{tab:flexetc} suggests
the real rank for each of the $64$ 
sextics listed in Section~2. In each case,
we report our best guess on the lower bound, based on numerical
experiments with that instance. Obtaining these numbers
with high confidence proved to be difficult.
We had considerable help from Anna Seigal and Emanuele Ventura
in carrying this out. 
Most puzzling is the real rank $20 $
we found for our representative of type
$(21)2{\rm nd}$, as this seems to contradict
\cite[Proposition 6.3]{MMSV}.
 This is either an error arising in our numerical method,
or the sextic lies on some exceptional locus.
Clearly, some further study is needed.

We did not yet attempt the same calculation for 
a larger sample of sextics in each rigid isotopy class.
This would be a very interesting future project at the
interface of numerics and real algebraic geometry.
The guiding problem  is to find the maximal
generic real rank among sextics.
To underscore the 
challenge, here is another open question:
the real rank of the monomial $x^2 y^2 z^2$
is presently unknown. It is either $11$, $12$ or $13$,
by \cite[Example 6.7]{MMSV}.

\smallskip

We now shift gears and turn to the construction of real K3 surfaces.
Two basic models of algebraic K3 surfaces
are quartic surfaces in $\PP^3$ and double-covers
of $\PP^2$ branched at a sextic curve. Thus, each
of our ternary sextics in Section 2 represents
a K3 surface over $\QQ$. 
Suppose we can write $ f = v_3^2-v_2 v_4$ where $v_i$ is a
form of degree $i$ in $x,y,z$.
Then $F=v_2 w^2+2 v_3 w + v_4$ is a quartic in four variables
that realizes the K3 surface with one singular point at $(0:0:0:1)$.
Blowing up that singular point gives the
 K3 surface encoded by $f$. Perturbing the coefficients of $F$
 gives a smooth quartic surface with similar properties.

The topology of the real surface
$V_\RR(F)$ is determined by the topological type of the real curve $V_\RR(f)$ and its sign behavior.
By perturbing $F$ to a polynomial $\tilde F$, we can obtain
a smooth quartic surface  whose real part $V_\RR(\tilde F)$ has the desired topology. See \cite{utkinquartic} for  details.

 The construction methods in Section \ref{sec:construct}
 reveal that many of the $64$ types can be 
 realized by adding a positive sextic to the 
 product of a quartic and a conic. 
 For such types, the sextic has the desired form
 $ f = v_3^2-v_2 v_4$. The resulting quartic $F$ has nice coefficients in $\QQ$.
 
 The real part of a smooth K3 surface is always an orientable surface. It has at most one connected component with 
 nonpositive Euler characteristic --- and therefore is determined (up to homeomorphism) by its total Betti number and its 
 Euler characteristic --- except when it is the union of two tori.  If it is nonempty, by Smith-Thom inequality its total Betti 
 number ranges between $2$ and $24$, and according to the Comessatti inequalities its Euler characteristic ranges 
 between $-18$ and $20$.
 There are $64$ possible combinations of these two numbers;
 they are displayed in \cite[Table (3.3), page 189]{silholbook}.
  All these $64$ possibilities can be realized as a quartic surface in $\PP^3$. These topological classification
  was studied by Utkin \cite{utkinquartic}.  The isotopic and rigid isotopic classifications are due to  Kharlamov \cite{Kha1,Kha2}.
  For proofs and further information we refer to Silhol's book \cite[Section VIII.4]{silholbook}.  We conclude by presenting two 
 explicit quartic surfaces that realize the minimal and the maximal Euler characteristic.

\begin{example} \rm Consider the smooth quartic surface $V_\CC(\bar F) \subset
\PP^3$ defined by the polynomial
$$ \begin{matrix} 
\bar F &=&  100 w^4 - 12500 w^2 x^2 + 104 x^4 - 12500 w^2 y^2 + 1640 x^2 y^2 + 
 1550 y^4 
  + 12500 w^2 y z \\ & & - 75 x^2 y z - 1552 y^3 z   + 9375 w^2 z^2 - 
 487 x^2 z^2 - 1533 y^2 z^2 + 354 y z^3 + 314 z^4. 
 \end{matrix}
 $$
Its real locus $V_\RR(\bar F)$ is a connected orientable surface of genus $10$. The
 Euler characteristic of that surface is $-18$. This is the smallest possible 
 Euler characteristic for a 
real K3 surface.
 
 We constructed the quartic $\tilde F$ from a sextic $f$ with ten non-nested ovals.
 Namely, $f = v_3^2 - v_2 v_4$, where
  $v_4 = 33001 x^4 + 131227 x^2 y^2 + 30980 y^4 - 11842 x^2 y z -  62072 y^3 z - 
  155986 x^2 z^2 - 122652 y^2 z^2 + 56672 y z^3 +  100672 z^4$, $v_3 = 10^{-3}z^3$ 
  and $v_2 = -4 x^2 - y^2 + 2 y z + 3 z^2$. Our code
  {\tt SexticClassifier} easily confirms that $V_\RR(f)$ has Type 10.
  The quartic $  F = v_2 w^2 + 2 v_3 w + v_4$ has a node at $(0:0:0:1)$.
  The projection from this node is ramified at $V_\CC(f)$.
  The K3 surface defined by $\tilde F = F   +\epsilon w^4$  has the desired properties
  for $\epsilon=10^{-10}$. Starting from $\tilde F$, we constructed $\bar F$
  using the techniques discussed in Section \ref{sec:construct} for
   improving integer coefficients.
 \end{example}

Our final example is dedicated to
the algebraic geometer Karl Rohn,
whose article \cite{Rohn}
inspired this project. Rohn
 was a professor
at the University of Leipzig from 1904 until 1920.

\begin{example} \rm 
 We start with Rohn's imaginary symmetroid in \cite[\S 9]{Rohn}. This is the quartic
 $$ G \quad = \quad \tau(s_1^2 - 6 s_2)^2 + (s_1^2 - 4 s_2)^2 - 64 s_4,$$ 
 where $s_i$ is the $i$th elementary symmetric polynomial in $x,y,z,w$
  and $\tau=(16\sqrt{10}-20)/135$. This is a nonnegative form with exactly $10$ real zeros. Subtracting a positive definite form multiplied with a small postive scalar gives a quartic surface with ten connected components. 
Using the techniques in
 Section \ref{sec:construct} we get the following  quartic with nice 
 integer coefficients:
 $$ \bar G \quad = \quad 6 s_1^4 - 53 s_1^2 s_2 + 120 s_2^2 - 320 s_4.$$
The surface $V_\RR(\bar G)$ is the disjoint union of ten spheres, so it has
Euler characteristic $20$.
\end{example}

\bigskip \bigskip

\begin{small}
\noindent
{\bf Acknowledgements.}
We are grateful to Paul Breiding, Claus Scheiderer, Anna Seigal,
Israel Vainsencher, Emanuele Ventura and Oleg Viro for their help.
This project was completed following a visit by Daniel Plaumann to MPI
Leipzig.  Nidhi Kaihnsa and Mahsa Sayyary Namin were funded by the
International Max Planck Research School {\em Mathematics in the
  Sciences} (IMPRS). Daniel Plaumann was supported through DFG grant
PL 549/3-1.  Bernd Sturmfels acknowledges  support by the US
National Science Foundation (DMS-1419018) and the Einstein Foundation
Berlin.
\end{small}

 \begin{small}
 
 \end{small}

\bigskip

\noindent
\footnotesize {\bf Authors' addresses:} \smallskip \\
\noindent Nidhi Kaihnsa, MPI Leipzig, \texttt{Nidhi.Kaihnsa@mis.mpg.de} \\
 Mario Kummer, MPI Leipzig, \texttt{Mario.Kummer@mis.mpg.de} \\
Daniel Plaumann, TU Dortmund, \texttt{Daniel.Plaumann@math.tu-dortmund.de} \\
Mahsa Sayyary Namin, MPI Leipzig, \texttt{Mahsa.Sayyary@mis.mpg.de} \\
Bernd Sturmfels, MPI Leipzig,  \texttt{bernd@mis.mpg.edu}  \
and UC Berkeley, \texttt{bernd@berkeley.edu}

\medskip

\noindent
Max-Planck Institute for Mathematics in the Sciences,
Inselstra\ss e 22, 04103 Leipzig, Germany \ \ \

\noindent 
Technische Universit\"at Dortmund, Fakult\"at f\"ur Mathematik,
44227 Dortmund, Germany

\noindent
Department of Mathematics, University of California, Berkeley, CA 94720, USA

\end{document}